\documentclass{amsart}
\usepackage[toc,page]{appendix}
\usepackage[headheight=12.0pt]{geometry}
\usepackage[all,cmtip]{xy}
\usepackage{amsmath}
\usepackage{amsfonts}
\usepackage{amssymb}
\usepackage{mathrsfs}   
\usepackage{amsthm}
\usepackage{xcolor}
\usepackage{cite}
\usepackage{stmaryrd}
\usepackage{graphicx}
\usepackage{pgf}
\usepackage{eepic}
\usepackage{pstricks}
\usepackage{epsfig}
\usepackage{fancyhdr}
\usepackage{tikz}
\usepackage[backref=page]{hyperref}

\hypersetup{
     colorlinks   = true,
     citecolor    = red,
     linkcolor = blue,
     urlcolor = purple
}
    
 \usepackage{mathptmx}  
  \newcommand\imCMsym[4][\mathord]{%
  \DeclareFontFamily{U} {#2}{}
  \DeclareFontShape{U}{#2}{m}{n}{
    <-6> #25
    <6-7> #26
    <7-8> #27
    <8-9> #28
    <9-10> #29
    <10-12> #210
    <12-> #212}{}
  \DeclareSymbolFont{CM#2} {U} {#2}{m}{n}
  \DeclareMathSymbol{#4}{#1}{CM#2}{#3}
}
\newcommand\alsoimCMsym[4][\mathord]{\DeclareMathSymbol{#4}{#1}{CM#2}{#3}}

\imCMsym{cmmi}{124}{\CMjmath}
\imCMsym[\mathop]{cmsy}{113}{\CMamalg}
\imCMsym[\mathop]{cmex}{96}{\CMcoprod}
\alsoimCMsym[\mathop]{cmex}{97}{\CMbigcoprod}
 
\theoremstyle{plain}
\newtheorem*{theoremu}{Theorem}
\newtheorem{theorem}{Theorem}[section]

\newtheorem{proposition}[theorem]{Proposition}
\newtheorem{corollary}[theorem]{Corollary}

\newtheorem{lemma}[theorem]{Lemma}

\theoremstyle{definition}
\newtheorem*{definitionu}{Definition}

\newtheorem{definition}[theorem]{Definition}

\theoremstyle{remark}
\newtheorem{remark}[theorem]{Remark}

\newcommand{\N}{{\mathbb N}}
\newcommand{\Z}{{\mathbb Z}}
\newcommand{\Q}{{\mathbb Q}}

\newcommand{\C}{{\mathbb C}}

\newcommand{\F}{{\mathbb F}}

\renewcommand{\P}{{\mathbb P}}

\newcommand{\D}{{\mathbb D}}

\newcommand{\lser}[1]{(\!(#1)\!)}
\newcommand{\pow}[1]{\llbracket #1 \rrbracket}

\newcommand{\spec}[1]{\mathrm{Spec}\left(#1\right)}
\newcommand{\cur}[1]{\mathcal{#1}}

\newcommand{\norm}[1]{\left\vert#1\right\vert}
\newcommand{\isomto}{\overset{\sim}{\rightarrow}}
\newcommand{\bu}{\bullet}

\newcommand{\cris}{\mathrm{cris}}
\newcommand{\rig}{\mathrm{rig}}

\newcommand{\ek}{\cur{E}_K}
\newcommand{\ekd}{\cur{E}_K^\dagger}
\newcommand{\tate}[1]{\langle #1 \rangle}

\newcommand{\spa}[1]{\mathrm{Spa}\left(#1\right)}
\newcommand{\spf}[1]{\mathrm{Spf}\left(#1\right)}

\newcommand{\et}{\mathrm{\acute{e}t}}

\newcommand{\pn}{(\varphi,\nabla)}
\newcommand{\rk}{\cur{R}_K}
\newcommand{\isoc}[1]{\mathrm{Isoc}^\dagger(#1)}
\newcommand{\fisoc}[1]{F\text{-}\mathrm{Isoc}^\dagger(#1)}
\newcommand{\pnm}[1]{\underline{\mathbf{M}\Phi}_{#1}^{\nabla}}

\title{Fundamental groups and good reduction criteria for curves over positive characteristic local fields}

\author{Christopher Lazda}
       \address{Dipartimento di Matematica ``Tullio Levi-Civita'' \\
       Via Trieste, 63 \\ 
        35121 Padova \\ 
        Italia}
       \email{lazda@math.unipd.it}

\begin{document}

\begin{abstract} In this article I define and study the overconvergent rigid fundamental group of a variety over an equicharacteristic local field. This is a non-abelian $\pn$-module over the bounded Robba ring $\ekd$, whose underlying unipotent group (after base changing to the Amice ring $\cur{E}_K$) is exactly the classical rigid fundamental group. I then use this to prove an equicharacteristic, $p$-adic analogue of Oda's theorem that a semistable curve over a $p$-adic field has good reduction iff the Galois action on its $\ell$-adic unipotent fundamental group is unramified.
\end{abstract}

\maketitle 

\tableofcontents

\section*{Introduction}

Let $F$ be a complete, discretely valued field of characteristic $p>0$, $R$ its ring of integers, and $k$ its residue field, which to begin with will be assumed finite. Choose a uniformiser $t$, thus $R\cong  k\pow{t}$ and $F\cong k\lser{t}$. Let $G_F=\mathrm{Gal}(F^\mathrm{sep}/F)$ denote the absolute Galois group of $F$. While for any prime $\ell\neq p$ the category $\mathrm{Rep}_{\Q_\ell}(G_F)$ of $\ell$-adic Galois representations is reasonably tractable (for instance, one has a good notion of local monodromy), the category of $p$-adic Galois representations of $G_F$ exhibits rather wild behaviour. Similarly, when one studies the cohomology of varieties over $F$, one finds that the $\ell$-adic \'etale theory is much better behaved than the $p$-adic one.

The way that this gap `at $p$' has generally been filled is by various versions of de\thinspace Rham or crystalline type cohomology; generally speaking you take a lift of your variety to characteristic $0$, then use the de Rham cohomology of this lift to define the $p$-adic cohomology of the original variety in characteristic $p$. The coefficient objects for this theory, analogous to $\ell$-adic \'etale sheaves, are then vector bundles with integrable connection, plus a Frobenius structure. 

One good candidate, then, for a better behaved $p$-adic analogue of $\mathrm{Rep}_{\Q_\ell}(G_F)$, has been the category of $\pn$-modules over the Robba ring $\rk$. To describe these objects properly, let $\cur{V}$ be a complete DVR with residue field $k$ and fraction field $K$ of characteristic $0$. Choose a uniformiser $\pi$ of $\cur{V}$, the residue field from now on assumed to be simply perfect.

\begin{definitionu} The Robba ring $\rk$ is the ring of analytic functions which converge on some half-open annulus $\{\eta \leq \norm{t} <1  \}$ over $K$. In other words, it consists of series $\sum_ia_it^i$ such that $\norm{a_i}\rho^i\rightarrow 0$ as $i\rightarrow \infty$ for \emph{all} $\rho<1$, and $\norm{a_i}\eta^i\rightarrow 0$ as $i\rightarrow -\infty$ for \emph{some} $\eta<1$.
\end{definitionu}

Let $\rk^\mathrm{int}$ denote the subring of series $\sum_ia_it^i\in \rk$ with $a_i\in \cur{V}$, one checks that $\rk^\mathrm{int}/(\pi)\cong k\lser{t}\cong  F$, therefore $\rk$ is in some sense a `lift' of $F$ to characteristic $0$.

\begin{definitionu} Fix some $q=p^a$ such that $k\supset \F_q$. A Frobenius $\sigma$ on $\rk$ is a continuous endomorphism, preserving $\rk^{\mathrm{int}}$, and inducing the absolute $q$-power Frobenius on $F$.
\end{definitionu}

The notion of a $\pn$-module will be needed in other situations than over $\rk$, so let me make the following, somewhat imprecise definition.

\begin{definitionu}  Let $B$ be a $K$-algebra together with a Frobenius $\sigma:B\rightarrow B$ and a $K$-derivation $\partial_t:B\rightarrow B$. Then a $\pn$-module over $B$ is a finite free $B$-module $M$, together with:
\begin{itemize} \item a connection $\nabla:M\rightarrow M$, i.e. a $K$-linear map such that $\nabla(bm)=\partial_t(b)m+b\nabla(m)$ for all $b\in B$, $m\in M$;
\item a horizontal isomorphism $\varphi:\sigma^*M\rightarrow M$, i.e. a morphism commuting with the connections $\sigma^*\nabla$ and $\nabla$.
\end{itemize}
Denote the category of $\pn$-modules over $B$ by $\pnm{B}$.
\end{definitionu}

One of the reasons that $\underline{\mathbf{M}\Phi}_{\rk}^{\nabla}$ is often considered a good $p$-adic analogue of $\mathrm{Rep}_{\Q_\ell}(G_F)$ is that there exists a local monodromy theorem for objects in $\pnm{\rk}$, analogous to Grothendieck's local monodromy theorem for $\ell$-adic Galois representations. In fact, the $p$-adic local monodromy theorem holds for an arbitrary perfect residue field $k$. This then begs the question of the existence of a well behaved (i.e. extended Weil) cohomology theory
\[ X\mapsto H^i_\rig(X/\rk) \]
for varieties over $F$, taking values in the category of $\pn$-modules over $\rk$.

In fact, as we argued in the book \cite{LP16}, there is a slightly more refined analogue of the category $\pnm{\rk}$ which should be considered the natural candidate for such a cohomology theory, and this is the category of $\pn$-modules over the \emph{bounded} Robba ring $\ekd\subset \rk$, consisting of those series $\sum_ia_it^i$ such that $\sup_i\norm{a_i}<\infty$. (The Frobenius $\sigma$ on $\rk$ will be assumed to satisfy $\sigma(\ekd)\subset \ekd$, hence the above definition applies.) While this is not the place to go into the details of why, one way to think of the relationship between $\pnm{\ekd}$ and $\pnm{\rk}$ is analogous to the relationship between the category of de Rham $p$-adic representations of the Galois group of a $p$-adic local field and the category of Weil--Deligne representations; unlike the $\ell$-adic case there is genuine information lost in passing from the former to the latter.

The construction of such a cohomology theory for varieties over $F$, with values in $\pnm{\ekd}$, was the main topic of the book \cite{LP16}, where cohomology groups $H^i_\rig(X/\ekd)$ were defined, and a base change result proved, showing that there is a natural isomorphism
\[ H^i_\rig(X/\ekd)\otimes_{\ekd} \ek \isomto H^i_\rig(X/\ek) \]
to classical rigid cohomology over the Amice ring $\ek$. (This ring, which arises as the $p$-adic completion of $\ekd$, is a complete discretely valued field with residue field $F$, hence Berthelot's construction of rigid cohomology applies.) Again, as we argued there, this base change result can be viewed as an equicharacteristic analogue of the fact that $p$-adic representations over a mixed characteristic local field `coming from geometry' are de Rham, or equivalently potentially semistable. In \cite{LP16} we also gave some arithmetic applications of this theory, for example we proved versions of the weight monodromy conjecture and the N\'eron--Ogg--Shafarevich criterion.

The purpose of this article is twofold: first of all, to extend the main results of \cite{LP16} from cohomology to the simplest non-abelian homotopical invariants of varieties over $F$, that is their unipotent fundamental groups, and secondly to use this fundamental group to give a homotopical criterion for good reduction of a semistable curve.

For the first goal then, what is needed is to construct a some `unipotent non-abelian $\pn$-module' $\pi_1^\rig(X/\ekd,x)$ (in a sense to be defined) whose underlying unipotent group scheme over $\ekd$, upon base changing to $\ek$, is canonically isomorphic to the classical rigid fundamental group $\pi_1^\rig(X/\ek,x)$. The construction of such an object is entirely similar to that of a `relative' fundamental group associated to a smooth and proper family over a curve achieved in \cite{Laz15}, and the base change result essentially follows immediately from base change in cohomology. To get the formalism to work properly, however, requires a few results to be proved about the cohomology theory $H^i_\rig(X/\ekd,-)$ beyond those from the book \cite{LP16}, for example stability of overconvergence by extensions, and rigidity of the various categories of isocrystals studied there.

I will also introduce a technical tool that permits a simplification of the arguments from \cite{Laz15}, namely what I call `absolute Frobenius cohomology'. This should be regarded as a $p$-adic analogue of Jannsen's continuous \'etale cohomology, and is constructed by promoting the cohomology groups $H^i_\rig(X/\ekd)$ to complexes in a suitable derived category, then taking derived global sections within that category. After recalling some of the basic formalism on Tannakian categories I then prove the first fundamental result, Theorem \ref{bc}, showing that there is a canonical $\pn$-module structure on the $\ekd$-valued rigid fundamental group $\pi_1^\rig(X/\ekd,x)$. I will also prove a comparison result with the global situation handled in \cite{Laz15}.

The second goal is then to prove an equicharacteristic analogue of the main result of \cite{AIK15}, giving a non-abelian criterion for the good reduction of a stable curve of genus $g\geq 2$. It has long been known that for higher genus curves over local fields, unramifiedness of the Galois action on \'etale cohomology is not sufficient to detect good reduction of the curve. Remarkably, however, Oda showed in \cite{Oda95} that this can (almost) be fixed by instead looking at the unipotent fundamental group. I say almost, because his result does not apply to arbitrary curves, but those which admit semistable reduction. For curves over the $p$-adic field $K$, this amounts to the existence of a regular scheme $\cur{X}$, flat and proper over the ring of integers $\cur{V}$, with special fibre $\cur{X}_0$ a reduced divisor with normal crossings. In this case, by blowing down certain `superfluous' components of the special fibre one arrives at a \emph{stable} model $\cur{X}'$ of $X$ over $\cur{V}$ (for a precise definition of stability, see Definition \ref{defn: stable}). Oda's result then states that this stable model $\cur{X}'$ is smooth over $\cur{V}$ if and only if the induced outer $G_K$ action on $\pi_1^\et(X_{\overline{K}})$ is trivial, and moreover that it suffices to look at the quotient $\pi_1^\et(X_{\overline{K}})^{(4)}$ by the $4$th term in the lower central series. In particular, a curve admitting semistable reduction has good reduction if and only if the Galois action on its unipotent fundamental group is unramified.

Andreatta--Iovita--Kim's theorem \cite{AIK15} is then a $p$-adic analogue of this fact, in which a rational point $x\in X(K)$ is fixed, and the \'etale fundamental group $\pi_1^\et(X_{\overline{K}},x)$ is replaced by its $p$-adic unipotent envelope $\pi_1^\et(X_{\overline{K}},x)_{\Q_p}$, considered as a non-abelian Galois representation. Their result then states that a given stable model of $X$ over $\cur{V}$ is smooth if and only if $\pi_1^\et(X_{\overline{K}},x)_{\Q_p}$ is \emph{crystalline}; again that it suffices to consider the quotient $\pi_1^\et(X_{\overline{K}},x)_{\Q_p}^{(4)}$. They prove their result essentially by reducing to a similar result of Oda over the complex numbers, deforming the mixed characteristic situation to $\C$ and identifying certain $p$-adic and topological monodromy operators.

Working $p$-adically and in equicharacteristic, i.e. with stable curves over $R$, the appropriate analogue of being `crystalline' for $\pn$-modules over $\ekd$, is being defined over $S_K=\cur{V}\pow{t}\otimes_\cur{V} K$. I will call such $\pn$-modules `non-singular', the terminology is meant to suggest that $t=0$ is a non-singular point of the  differential equation. There is also a logarithmic analogue, which I will call `regular' (i.e. having regular singularities) which corresponds to the notion of a Galois representation being `semistable'. For those $\pn$-modules which are already regular, non-singularity is then again equivalent to the vanishing of a certain monodromy operator. The criterion for good reduction is then the following.

\begin{theoremu}[\ref{main2}, \ref{main4}] Let $\cur{X}\rightarrow \spec{R}$ be a proper, stable curve of genus $g\geq 2$. Then $\cur{X}$ is smooth over $R$ if and only if the unipotent fundamental group $\pi_1^\rig(X/\ekd,x)$ is non-singular. In fact, it suffices to consider the quotient $\pi_1^\rig(X/\ekd)^{(4)}$ by the $4$th term in the lower central series.
\end{theoremu}

Just as the main theorem of \cite{AIK15} is proved by deforming to the complex situation and applying an analogous result of Oda \cite{Oda95}, the proof of Theorem \ref{main2} will proceed by deforming to mixed characteristic and applying the result of Andreatta--Iovita--Kim. As previously mentioned, in the semistable case, the `non-singularity' of $\pi_1^\rig(X/\ekd,x)$ is equivalent to the vanishing of a certain monodromy operator, this can now be detected entirely on the level of the special fibre $X_0$ of $\cur{X}$, considered as a log scheme. Deforming this log-smooth curve to mixed characteristic reduces to the situation considered in \cite{AIK15}. In fact, this last step could also be accomplished by using work in progress of Chiarellotto, di\thinspace Proietto and Shiho \cite{CDPS}, which will provide a completely algebraic proof of these sorts of `good reduction' theorems for curves. Their result should then apply directly to the log scheme $X_0$ in characteristic $p$.

\subsection*{Acknowledgements}

I would like to thank Bruno Chiarellotto and Ambrus P\'al for suggesting that the machinery of \cite{Laz15,LP16} could be used to prove an equicharacteristic analogue of Andreatta--Iovita--Kim's results, as well as for numerous useful conversations regarding the contexts of this article. In particular, I would like to thank Prof. Chiarellotto for explaining to me the right way to `relativise' the universal unipotent objects that give rise to the fundamental group. I was supported during the writing of this article by a Marie Curie fellowship of the Istituto Nazionale di Alta Matematica.

\subsection*{Notations and conventions}

I will denote by $F$ a complete discretely valued field of characteristic $p>0$, by $R$ its ring of integers, $t$ a uniformiser and $k$ its residue field, which will be assumed perfect. I will fix $q=p^a$ such that $k\supset \F_q$, and Frobenius will always refer to the $q$-power Frobenius. I will write $W=W(k)$ for the ring of Witt vectors of $k$, and $K_0=W[1/p]$ for its fraction field. I will denote by $K$ a finite, totally ramified extension of $K_0$, and $\cur{V}$ its ring of integers; $\sigma$ will denote a choice of ($q$-power) Frobenius on $\cur{V}$ or $K$, and $\pi$ a uniformiser of $\cur{V}$. I will write $K^\sigma$ for the fixed field of Frobenius, it is thus a finite totally ramified extension of $\Q_q:=W(\F_q)[1/p]$. Formal schemes will always be understood to be $p$-adic formal schemes, in particular the power series ring $\cur{V}\pow{t}$ will be equipped with the $p$-adic topology and not the $(p,t)$-adic one.

I will denote by $\rk$ the Robba ring over $K$, that is the ring of series $\sum_ia_it^i$ with $a_i\in K$, converging on some half open annulus $\{\eta\leq \norm{t} <1\}$, by $\ekd\subset \rk$ its bounded subring, i.e. the subring of elements such that $\sup_i\norm{a_i}<\infty$, and by $\rk^{\mathrm{int}}=\mathcal{O}_{\ekd}\subset \rk$ the subring of integral elements, i.e. those with $a_i\in \cur{V}$. I will let $\rk^+$ denote the plus part of $\rk$, that is the subring of elements with $a_i=0$ for $i<0$, and write $S_K=\rk^+\cap \ekd$, this is equal to $\cur{V}\pow{t}\otimes_\cur{V} K$. 

The ring $\rk^\mathrm{int}=\mathcal{O}_{\ekd}$ is a Henselian local ring, with maximal ideal generated by $\pi$ and fraction field $\ekd$, I will let $\ek$ denote the completion of $\ekd$ and $\mathcal{O}_{\ek}$ its ring of integers. Equivalently, $\ek$ is the ring of series $\sum_ia_it^i$ with $\sup_i \norm{a_i}<\infty$ and $\norm{a_i}\rightarrow 0$ as $i\rightarrow -\infty$, and $\mathcal{O}_{\ek}$ is the subring of $\ek$ consisting of series with $a_i\in \cur{V}$. I will fix once and for all a Frobenius $\sigma$ on $\cur{V}\pow{t}$, that is a $p$-adically continuous endomorphism, $\sigma$-linear over $\cur{V}$, and lifting the absolute $q$-power Frobenius on $\cur{V}\pow{t}/(\pi)\cong A$. Such a Frobenius extends uniquely to a Frobenius on each of $\mathcal{O}_{\ekd}, \mathcal{O}_{\ek}, S_K,\rk^+,\ekd,\ek$ and $\rk$, all of which I will also denote by $\sigma$. I will further assume that $\sigma(t)=ut^q$ for some $u\in S_K^\times$ such that $u\equiv 1 \text{ mod }\pi$.

\section{Fundamental groups for varieties over Laurent series fields}\label{parti}

Let $X/F$ be a geometrically connected variety, and suppose that there is some rational point $x\in X(F)$. To define $p$-adic fundamental groups, it will be necessary to use Berthelot's category $\mathrm{Isoc}^\dagger(X/\ek)$ of overconvergent isocrystals on $X/\ek$, introduced in \cite[\S2]{Ber96b} as a $p$-adic analogue of the category of lisse $\ell$-adic sheaves on $X_{\overline{F}}$. To define these objects, one takes an embedding $X\hookrightarrow \mathfrak{P}$ into a smooth and proper formal $\cur{O}_{\ek}$-scheme, and considers the embedding 
\[ j:]X[_\mathfrak{P} \rightarrow ]\overline{X}[_\mathfrak{P} \]
of the tubes (here $\overline{X}$ is the closure of $X$ in $\mathfrak{P}$). Letting $j_X^\dagger\cur{O}_{]\overline{X}[_\mathfrak{P}}$ denote the subsheaf of $j_*\cur{O}_{]X[_\mathfrak{P}}$ consisting of functions which converge on some strict neighbourhood of $]X[_\mathfrak{P}$,  an overconvergent isocrystal is then a coherent $j_X^\dagger\cur{O}_{]\overline{X}[_\mathfrak{P}}$-module with integrable connection, together with a convergence condition on its Taylor series. One can check that this does not depend on the choice of embeddings, and hence can be glued if $X$ does not globally admit such an embedding.

Both $\mathrm{Isoc}^\dagger(X/\ek)$ as well as its full subcategory $\cur{N}\mathrm{Isoc}^\dagger(X/\ek)$ of unipotent objects (i.e. those which are iterated extensions of the unit object) are Tannakian, and pulling back via the point $x$ provides a fibre functor
\[  x^*:\cur{N}\mathrm{Isoc}^\dagger(X/\ek)\rightarrow  \mathrm{Vec}_{\ek}.\]

\begin{definition} The (unipotent) rigid fundamental group  $\pi_1^\rig(X/\ek,x)$ of $X/\ek$ is by definition the affine group scheme representing automorphisms of the fibre functor $x^*$.
\end{definition}

There is therefore a canonical equivalence of categories $\mathrm{Rep}_K(\pi_1^\rig(X/\ek,x))\cong \cur{N}\mathrm{Isoc}^\dagger(X/\ek)$ such that the diagram
\[ \xymatrix{ \mathrm{Rep}_K(\pi_1^\rig(X/\ek,x))\ar[rr]\ar[dr]_{\mathrm{forget}} & & \cur{N}\mathrm{Isoc}^\dagger(X/\ek) \ar[dl]^{x^*} \\ & \mathrm{Vec}_K }\]
commutes. Note that $\pi_1^\rig(X/\ek,x)$ should be considered as a $p$-adic analogue of the \emph{geometric} unipotent $\ell$-adic fundamental group $\pi_1^\et(X_{\overline{F}},x)_{\Q_\ell}$ classifying unipotent lisse $\Q_\ell$-sheaves on $X_{\overline{F}}$ (for $\ell\neq p$). This latter object comes with the extra structure of a Galois action, and the analogous extra structure in the $p$-adic world is that of a $\pn$-module over $\ek$. In the non-abelian case, this really means that the Hopf algebra $\hat{A}^{\rig,\vee}_{\infty,x}:=\mathcal{O}(\pi_1^\rig(X/\ek,x)$ is a `Hopf algebra in the category $\underline{\mathbf{M}\Phi}^\nabla_{\ek}$ of $\pn$-modules over $\ek$' in the following sense.

\begin{definition} \label{hopfenriched}  An Hopf algebra in $\underline{\mathbf{M}\Phi}^\nabla_{\ek}$ is an ind-object $A\in \mathrm{Ind}(\underline{\mathbf{M}\Phi}^\nabla_{\ek})$ together with maps 
\begin{itemize}\item $u:\ek\rightarrow A$ (unit),
\item $\epsilon:A\rightarrow \ek$ (augmentation),
\item $m:A\otimes_{\ek} A\rightarrow A$ (multiplication),
\item $c:A\rightarrow A\otimes_{\ek} A$ (comultiplication),
\item $i:A\rightarrow A$ (antipode)
\end{itemize}
making the diagrams expressing the axioms for a Hopf algebra commute.
\end{definition}

As a concrete example, the `associativity' of the multiplication is expressed by the commutativity of the diagram
\[ \xymatrix{ A\otimes_{\ek}A\otimes_{\ek}A \ar[r]^-{\mathrm{id}\otimes m} \ar[d]_-{m\otimes\mathrm{id}} & A\otimes_{\ek} A\ar[d]^-m \\ A\otimes_{\ek} A \ar[r]^-m & A. }\]
It is fairly easy to see that the `forgetful functor' $\underline{\mathbf{M}\Phi}^\nabla_{\ek}\rightarrow \mathrm{Vec}_K$ takes Hopf algebras to Hopf algebras, and so the following definition makes sense.

\begin{definition} A $\pn$-module structure on $\pi_1^\rig(X/\ek,x)$ is a Hopf algebra $\hat{\cur{A}}^{\rig,\vee}_{\infty,x}$ in $\underline{\mathbf{M}\Phi}^\nabla_{\ek}$ whose underlying Hopf algebra over $\ek$ is $\hat{A}^{\rig,\vee}_{\infty,x}$
\end{definition}

I won't go into how to construct such a $\pn$-module structure on $\pi_1^\rig(X/\ek,x)$, since I will do this in detail for the more refined `overconvergent version' which lies over $\ekd$, and the construction is essentially identical. 

As alluded to in the introduction, (one of) the main point(s) of the book \cite{LP16} was to show that the rigid cohomology $H^i_\rig(X/\ek)$ of a variety $X/F$, which is a $\pn$-module over $\ek$, descends to a $\pn$-module $H^i_\rig(X/\ekd)$ over the bounded Robba ring $\ekd$ (by Kedlaya's full faithfulness theorem, this latter object is unique up to isomorphism). The first main aim of this article is to extend this to the unipotent fundamental group. In \cite{LP16} we introduced a category of $\ekd$-valued isocrystals $\mathrm{Isoc}^\dagger(X/\ekd)$, and using this instead of $\mathrm{Isoc}^\dagger(X/\ek)$ defines the unipotent fundamental group $\pi_1^\rig(X/\ekd,x)$ as an affine group scheme over $\ekd$. To put a $\pn$-module structure on this group, I will use essentially the same methods as in \cite{Laz15}, considering a certain `relative Tannakian' situation; specifically I will use the category of `absolute isocrystals' $\fisoc{X/K}$ and the natural functors
\[ H^i_\rig(X/\ekd,-):\fisoc{X/K} \leftrightarrows \pnm{\ekd}:-\otimes \mathcal{O}_{X/K}^\dagger , \]
the key point being to prove a certain base change theorem (Theorem \ref{bc}). To make the arguments from \cite{Laz15} slightly more streamlined, I will introduce a version of `absolute Frobenius cohomology', which can be viewed as an equicharacteristic analogue of syntomic cohomology for varieties over finite extensions of $\Q_p$. While not strictly necessary to make the argument work (for instance, the corresponding theorem in \cite{Laz15} was proved without any reference to an absolute Frobenius cohomology) it does make it simpler. Moreover, the absolute Frobenius cohomology groups introduced will likely have numerous future applications, for instance in the study of $L$-functions attached to isocrystals, and this potential more than justifies the (mostly formal) extra work required in setting the theory up.

The end product of all this, then, is a non-abelian $\pn$-module $\pi_1^\rig(X/\ekd,x)$ over $\ekd$, in other words a Hopf algebra in the category $\pnm{\ekd}$. The analogue of the base change results from \cite{LP16} is then the fact that $\pi_1^\rig(X/\ekd,x)\otimes_{\ekd} \ek \cong \pi_1^\rig(X/\ek,x)$, I will prove this in \S\ref{maini}. I will also prove a comparison result with the global case, i.e. with the relative fundamental group of a smooth and proper family $\mathcal{X}\rightarrow C$ constructed in \cite{Laz15}.

\section{Preliminaries on overconvergent isocrystals}

In the book \cite{LP16} we introduced categories $(F\text{-})\mathrm{Isoc}^\dagger(X/\ekd)$ and $(F\text{-})\mathrm{Isoc}^\dagger(X/K)$ of overconvergent ($F$-)isocrystals on a variety $X/F$, relative to $\ekd$ and $K$ respectively, refining the classical categories of overconvergent ($F$-)isocrystals relative to $\ek$. Using a Tannakian formalism to study the fundamental group $\pi_1^\rig(X/\ekd,x)$ requires certain results about these categories, the most important of which (clearly) is that they are in fact Tannakian, although there are others. In this section I will gather together some of these results needed in the rest of the article. 

\begin{lemma} \label{ef} Let $(X,Y,\mathfrak{P})$ be a smooth frame over $\cur{V}\pow{t}$, with base change $(X,Y',\mathfrak{P}')$ to $\mathcal{O}_{\ek}$. Then the functor $\mathrm{MIC}^\dagger((X,Y,\mathfrak{P})/\ekd)\rightarrow \mathrm{MIC}^\dagger((X,Y',\mathfrak{P}')/\ek)$ is exact and faithful.
\end{lemma}

\begin{proof} Faithfulness follows from \cite[Lemma 2.59]{LP16}. To show exactness, the question is local on $\mathfrak{P}$, which I may therefore assume to be affine, and on $X$, which I may therefore assume to be of the form $X=Y\cap D(g)$ for some $g\in \cur{O}_\mathfrak{P}$. Let $f_1,\ldots,f_r\in \cur{O}_\mathfrak{P}$ be such that $Y=P\cap V(f_1,\ldots,f_r)$ (here $P$ is the closed fibre of $\mathfrak{P}$.

The exactness question being local on $]Y[_\mathfrak{P}$, it suffices to work instead on each of the quasi-compact tubes
\[ [Y]_n:=\left\{\left. x\in \mathfrak{P}_K\right\vert v_x(\pi^{-1}f_i^n)\leq 1\;\forall i \right\} \]
which cover $]Y[_\mathfrak{P}$, note also that it suffices to treat the case of coherent $j_X^\dagger\cur{O}_{[Y]_n}$-modules (without integrable connection), i.e. to show that the functor
\[ \mathrm{Coh}\left( j_X^\dagger\cur{O}_{[Y]_n} \right) \rightarrow \mathrm{Coh}\left(j_X^\dagger\cur{O}_{[Y']_n} \right) \]
is exact. Let $V_{n,m}$ denote the open subset of $[Y]_n$ defined by 
\[ V_{n,m}:= \left\{\left. x\in [Y]_n \right\vert v_x(\pi^{-1} g^{m}) \geq 1 \right\}.\]
Write
\[ \cur{E}_m = \frac{S_K\tate{T}}{(t^mT-\pi)},\;\; \cur{O}_{\cur{E}_m} = \left(\frac{\cur{V}\pow{t}\tate{T}}{(t^mT-\pi)}\right)^\circ \]
where the $(-)^\circ$ means integral closure, so that $\spa{\cur{E}_m,\cur{O}_{\cur{E}_m}}$ is a finite type adic space over $\D^b_K$, and $V_{n,m}$ is of finite type over $\spa{\cur{E}_m,\cur{O}_{\cur{E}_m}}$. Let $V'_{n,m}=V_{n,m}\times_{\spa{\cur{E}_m,\cur{O}_{\cur{E}_m}}} \spa{\ek,\cur{O}_{\ek}}$. Then applying \cite[Lemma 2.66]{LP16}, together with its `classical' analogue, it suffices to show that the functor
\[ \mathrm{Coh}\left( \cur{O}_{V_{n,m}} \right) \rightarrow \mathrm{Coh}\left(\cur{O}_{V_{n,m}'} \right) \]
is exact. Write $A_{n,m}:=\Gamma(V_{n,m},\cur{O}_{V_{n,m}})$, $A'_{n,m}:=\Gamma(V'_{n,m},\cur{O}_{V'_{n,m}})$, then what I need to show then is that the functor $-\otimes_{A_{n,m}} A'_{n,m}$ is exact for finitely generated $A_{n,m}$-modules. 

Since it is clearly right exact, it suffices to show left exactness. Note that for coherent $A_{n,m}$-modules the functor $-\otimes_{A_{n,m}} A'_{n,m}$ can be identified with the completed tensor product functor $-\widehat{\otimes}_{A_{n,m}} A'_{n,m}$, and hence, since
\[ A'_{n,m}\cong A_{n,m}\widehat{\otimes}_{\cur{E}_m} \cur{E}_K, \]
with the completed tensor product functor $-\widehat{\otimes}_{\cur{E}_m} \cur{E}_K$. Next, define $B_{n,m}=\Gamma(V_{n,m},\cur{O}^+_{V_{n,m}})$, and similarly $B'_{n,m}=\Gamma(V'_{n,m},\cur{O}^+_{V'_{n,m}})$. Since the functor $-\otimes_{\Z_p} \Q_p$ inverting $p$ is exact, it suffices to show that the functor $ -\widehat{\otimes}_{\cur{O}_{\cur{E}_m}} \cur{O}_{\ek} $ is left exact on the exact (but not abelian) category of finitely generated, $p$-torsion free $B_{n,m}$-modules. Since this completed tensor product functor can now be written as
\[ M \widehat{\otimes}_{\cur{O}_{\cur{E}_m}} \cur{O}_{\ek} = \varprojlim_n\left( M/p^nM \otimes_{\cur{O}_{\cur{E}_m}/p^n} \cur{O}_{\ek}/p^n\right) \]
it simply suffices to note that projective limits are left exact, $\cur{O}_{\cur{E}_m}/p^n\rightarrow \cur{O}_{\ek}/p^n$ is flat, and $M\mapsto M/p^nM$ is exact on the category of $p$-torsion free $B_{n,m}$-modules.
\end{proof}

\begin{corollary} Let $X/F$ be geometrically connected, and $x\in X(F)$ a point. Then $\mathrm{Isoc}^\dagger(X/\ekd)$ is a neutral Tannakian category over $\ekd$, and 
\[ x^*:\mathrm{Isoc}^\dagger(X/\ekd)\rightarrow \mathrm{Vec}_{\ekd} \]
is a fibre functor.
\end{corollary}

\begin{proof} There is a commutative diagram
\[ \xymatrix{ \mathrm{Isoc}^\dagger(X/\ekd)\ar[r]^-{x^*}\ar[d] & \mathrm{Vec}_{\ekd}\ar[d] \\ \mathrm{Isoc}^\dagger(X/\ek) \ar[r]^-{x^*} & \mathrm{Vec}_{\ek}  }\]
with $\isoc{X/\ek}$ neutral Tannakian over $\ek$ and $x^*:\isoc{X/\ek}\rightarrow \mathrm{Vec}_{\ek}$ a fibre functor. Since both vertical arrows are faithful and exact (and commute with taking duals), the claim is now straightforward, once it has been verified that $\mathrm{End}_{\isoc{X/\ekd}}(\mathcal{O}_{X/\ekd}^\dagger)=\ekd$. This follows from the fact that the tube $]X[_\mathfrak{P}$ is connected.
\end{proof}

Now let $\cur{N}\mathrm{Isoc}^\dagger(X/\ekd)\subset \mathrm{Isoc}^\dagger(X/\ekd)$ denote the full subcartegory of unipotent objects, i.e. those which are iterated extensions of the unit object.

\begin{corollary} In the above situation, $\cur{N}\mathrm{Isoc}^\dagger(X/\ekd)$ is a neutral Tannakian category over $\ekd$, and 
\[ x^*:\cur{N}\mathrm{Isoc}^\dagger(X/\ekd)\rightarrow \mathrm{Vec}_{\ekd} \]
is a fibre functor.
\end{corollary}

Hence I may define the (unipotent) rigid fundamental group of $X/\ekd$ (at least as an affine group scheme over $\ekd$) as $\pi_1^\rig(X/\ekd,x):=\underline{\mathrm{Aut}}^\otimes(x^*)$. It will also be important to know that certain categories of `absolute' coefficients as introduced in \cite{LP16} are also Tannakian, but this time over $K$ (or subfields thereof).

\begin{corollary}\label{tann2}  Let $X/F$ be geometrically connected. Then the category $\isoc{X/K}$ is Tannakian over $K$, and $\fisoc{X/K}$ is Tannakian over $K^{\sigma}$.
\end{corollary}

\begin{proof} Since $(F\text{-})\isoc{X/K}\rightarrow \isoc{X/\ekd}$ is a faithful, exact, tensor functor, it suffices to calculate the endomorphism group of the unit object in both categories, which is $K$ for $\isoc{X/K}$ and $K^{\sigma}$ for $\fisoc{X/K}$.
\end{proof}

Another important result will be the fact that rigid cohomology $H^1_\rig(X/\ekd,-)$ computes the extension groups in $\mathrm{Isoc}^\dagger(X/\ekd)$. This is the essential content of the following result.

\begin{proposition} \label{exts} Let $(X,Y,\mathfrak{P})$ be a smooth frame. Then the full subcategories
\begin{align*} \mathrm{MIC}^\dagger((X,Y,\mathfrak{P})/\ekd) &\subset  \mathrm{MIC}((X,Y,\mathfrak{P})/\ekd)  \\
\mathrm{MIC}^\dagger((X,Y,\mathfrak{P})/K) &\subset  \mathrm{MIC}((X,Y,\mathfrak{P})/K) 
\end{align*}
of overconvergent connections are stable under extensions.
\end{proposition}

\begin{proof} The proof is essentially the same in both cases, I will therefore only give the version relative to $\ekd$, following the proof of the corresponding result in classical rigid cohomology given in \cite{CLS99a}. The question is local on $\mathfrak{P}$ and $X$, I may therefore assume the former to be affine, $\mathfrak{P}\cong \spf{A}$, and $X=D(g)\cap Y$ for some $g\in A$. Choose functions $t_i\in A$ such that the $dt_i$ form a basis for $\Omega^1_{A/\cur{V}\pow{t}}$ in a neighbourhood of $X$, and let $\partial_i$ be the dual basis of the tangent bundle. 

Now let $0 \rightarrow E_1 \rightarrow E_2 \rightarrow E_3\rightarrow 0$ be an exact sequence in $\mathrm{MIC}((X,Y,\mathfrak{P})/\ekd)$ with $E_1$ and $E_2$ overconvergent, and fix integers $0\leq n,m$. Let $V_{n,m}\subset \mathfrak{P}_K$ be the open subspaces as considered above in the proof of Lemma \ref{ef}, and let $U\subset V_{n,m}$ be an open affinoid such that $E_1$ and $E_3$, and hence $E_2$, are free on $U$. Choose a basis for $E_1$ and extend to a basis for $E_3$, and for each $i$ let 
\[ \left( \begin{matrix} M_i & N_i \\ 0 & P_i \end{matrix} \right) \in \Gamma(U,\cur{O}_U) \] be the matrix representing the action of $\partial_i$ with respect to this basis.

Next fix an integer $n'>n$, let $p_i:\mathfrak{P}_K^2\rightarrow \mathfrak{P}_K$ denote the two projections, and write $V^{(2)}_{n',m}=[Y]^{(2)}_{n'} \cap p_1^{-1}(U)$ where the tube $[Y]^{(2)}_{n'}$ is taken inside $\mathfrak{P}_K^2$ (embedding $Y$ via the diagonal). Let $\cur{P}$ be the structure sheaf of the formal completion of the diagonal $\mathfrak{P}_K$ inside $\mathfrak{P}_K^2$, so that there is a natural inclusion $\Gamma(V^{(2)}_{n',m},\cur{O}_{V^{(2)}_{n',m}})\rightarrow \Gamma(U,\cur{P})$ expressing the latter as the formal completion of the former. Finally let 
\[  \left( \begin{matrix} A & B \\ 0 & C \end{matrix} \right)\in \Gamma(U,\cur{P}) \]
be the matrix of the formal Taylor isomorphism $ \cur{P}\otimes E_2\isomto E_2\otimes\cur{P}$ on $U$. Then overconvergence of $E_1$ and $E_3$ implies that after possibly increasing $m$, $A,C\in \Gamma(V^{(2)}_{n',m},\cur{O}_{V^{(2)}_{n',m}})\subset \Gamma(U,\cur{P})$, I want to show that in fact $B\in \Gamma(V^{(2)}_{n,m},\cur{O}_{V^{(2)}_{n,m}})$, where $V^{(2)}_{n,m}=[Y]^{(2)}_{n} \cap p_1^{-1}(U)$, again taking the tube $[Y]^{(2)}_{n}$ inside $\mathfrak{P}_K^2$.

To see this, note that the action of $\partial_i$ on the matrix $ \left( \begin{matrix} A & B \\ 0 & C \end{matrix} \right)$ is simply given by left multiplication by $ \left( \begin{matrix} p_2^*M_i & p_2^*N_i \\ 0 & p_2^*P_i \end{matrix} \right)$, from which, as in the proof of \cite[Proposition 1.2.2]{CLS99a}, follows the identity
\[ \partial_i(A^{-1}B)=A^{-1}p_2^*N_i C. \]
Since the RHS has coefficients in $\Gamma(V^{(2)}_{n',m},\cur{O}_{V^{(2)}_{n',m}})$ and $n<n'$ it follows that $A^{-1}B$, and therefore $B$, has coefficients in $\Gamma(V_{n,m},\cur{O}_{V_{n,m}})$, as required.

Finally, by covering $V_{n,m}$ with such $U$ and letting $n\rightarrow \infty$ shows that the Taylor isomorphism of $E_2$ converges in a strict neighbourhood of the diagonal, and hence $E_2$ is overconvergent as claimed.
\end{proof}

\begin{corollary} Let $X/F$ be a variety, and $E_1,E_2 \in \mathrm{Isoc}^\dagger(X/\ekd)$. Then there is an isomorphism
$$ \mathrm{Ext}^1_{ \mathrm{Isoc}^\dagger(X/\ekd)}(E_1,E_2) \cong H^1_\rig(X/\ekd,E_1^\vee\otimes E_2)
$$
natural in $E_i$ and $X$.
\end{corollary}

\begin{proof} If $(X,Y,\mathfrak{P})$ is a smooth and proper frame, then $H^1_\rig(X/\ekd,E_1^\vee\otimes E_2)$, as a de Rham cohomology group, computes the extension group in the category $ \mathrm{MIC}((X,Y,\mathfrak{P})/\ekd)$, by the previous proposition this coincides with the extension group in the category $ \mathrm{Isoc}^\dagger(X/\ekd)$. The non-embeddable case is then deduced using Zariski descent.
\end{proof}

\begin{remark} If $E \in\cur{N}\isoc{X/\ekd}$, repeated applications of the five lemma show that $H^1_\rig(X/\ekd,E)$ is finite dimensional over $\ekd$.
\end{remark}

For `absolute' coefficient objects $E\in\fisoc{X/K}$ the cohomology groups $H^i_\rig(X/\ekd,E)$ play the role of the higher direct images of $E$ under the structure map $X\rightarrow \spec{F}$, via the identification
\[ \pnm{\ekd} \cong \fisoc{\spec{F}/K}  \]
from \cite{LP16}. There are a few results centred on this way of viewing things that will be needed, specifically an adjoint functor to $H^0_\rig(X/\ekd,-)$, a projection formula, and a `Leray spectral sequence'. The first two of these are easy.

\begin{proposition}\label{proj} The functor
\begin{align*} \pnm{\ekd}&\rightarrow \fisoc{X/K} \\
M&\mapsto M\otimes \mathcal{O}_{X/K}^\dagger \end{align*}
is left adjoint to
\begin{align*} \fisoc{X/K} &\rightarrow \pnm{\ekd} \\
E&\mapsto H^0_\rig(X/\ekd,E), \end{align*}
and for any $E\in \fisoc{X/K}$, $M\in \pnm{\ekd}$, and any $i\geq0$, there is a canonical isomorphism
\[ H^i_\rig(X/\ekd,E\otimes (M\otimes \mathcal{O}_{X/K}^\dagger)) \cong H^i_\rig(X/\ekd,E)\otimes_{\ekd} M \]
of $\pn$-modules.
\end{proposition}

\begin{remark} In the above situation, I will generally write $E\otimes M$ instead of $E\otimes (M\otimes \mathcal{O}_{X/K}^\dagger)$, thus the `projection formula' becomes
\[ H^i_\rig(X/\ekd,E\otimes M) \cong H^i_\rig(X/\ekd,E)\otimes_{\ekd} M.\]
\end{remark}

\begin{proof} Entirely straightforward.
\end{proof}

The Leray spectral sequence I will require will show how to calculate extension groups in the category $\fisoc{X/K}$, and will be part of a more general theory of `absolute' cohomology for overconvergent $F$-isocrystals. To set up this theory properly will require promoting the cohomology groups $H^i_\rig(X/\ekd,E)\in \pnm{\ekd}$ to a complex in a suitable derived category of $\pn$-modules, and this is the topic of the next section.

\section{Absolute Frobenius cohomology}\label{safc}

The purpose of this section is to introduce `absolute' cohomology groups for objects $E\in \fisoc{X/K}$, analogous to Jannsen's continuous $\ell$-adic cohomology in the \'etale theory, or Deligne cohomology in Hodge theory. The basic point is to promote the cohomology groups $H^i_\rig(X/\ekd,E)\in \pnm{\ekd}$ to an object in a suitable derived category, and then show that the hom groups in this category compute extension groups of $F$-isocrystals. The vast majority of this section is fairly formal, and while the reason for introducing these `absolute cohomology groups' is to make the proof of Theorem \ref{bc} a bit more streamlined, I fully expect there to be many other future applications, for example in the study of $L$-functions.

For simplicity I will assume throughout that there is a smooth and proper frame $(X,Y,\mathfrak{P})$ admitting a Frobenius lift $\sigma$, the general case is handled using descent methods and simplicial frames (the proofs are identical but the notation far more cumbersome). The first thing I need to do is relax the fairly stringent conditions defining the category $\fisoc{X/K}$.

\begin{definition} Let $\widetilde{\mathrm{MIC}}((X,Y,\mathfrak{P})/K)$ denote the category of (not necessarily coherent) $j_X^\dagger\mathcal{O}_{]Y[_\mathfrak{P}}$-modules with a (not necessarily overconvergent) integrable connection. 
\end{definition}

\begin{definition} A weak Frobenius structure on an object $E$ of $\widetilde{\mathrm{MIC}}((X,Y,\mathfrak{P})/K)$ is a morphism
\[ \sigma^*E\rightarrow E \]
in $\widetilde{\mathrm{MIC}}((X,Y,\mathfrak{P})/K)$. The category of modules with integrable connection and weak Frobenius is denoted $F^w\text{-}\widetilde{\mathrm{MIC}}((X,Y,\mathfrak{P})/K)$.
\end{definition}

There is therefore a fully faithful functor $\fisoc{X/K}\rightarrow F^w\text{-}\widetilde{\mathrm{MIC}}((X,Y,\mathfrak{P})/K)$. I will also write $\widetilde{\underline{\mathbf{M}}}^{\nabla}_{\ekd}$ for the category of $\ekd$-vector spaces with connection (not necessarily finite dimensional) and $\widetilde{\underline{\mathbf{M}\Phi}}^{w,\nabla}_{\ekd}$ for the category of $\ekd$-vector spaces with connection and (compatible) weak Frobenius. I do not claim that these are equivalent to the categories $(F^w\text{-})\widetilde{\mathrm{MIC}}((X,Y,\mathfrak{P})/K)$ over the `base' frame
\[ (\spec{F},\spec{R},\spf{\cur{V}\pow{t}}),\]
and I see no reason to believe that they are. There is, however, a fully faithful functor $\pnm{\ekd}\rightarrow \widetilde{\underline{\mathbf{M}\Phi}}^{w,\nabla}_{\ekd}$.

\begin{lemma} \label{injs} The categories $F^w\text{-}\widetilde{\mathrm{MIC}}((X,Y,\mathfrak{P})/K)$ and $\widetilde{\mathrm{MIC}}((X,Y,\mathfrak{P})/K)$ have enough injectives, and the forgetful functor
\[F^w\text{-}\widetilde{\mathrm{MIC}}((X,Y,\mathfrak{P})/K)\rightarrow \widetilde{\mathrm{MIC}}((X,Y,\mathfrak{P})/K) \]
preserves them.
\end{lemma}

\begin{proof} Note that the functor $E\mapsto \bigoplus_{n\geq0} (\sigma^*)^nE$ gives an exact left adjoint to the exact forgetful functor, and hence by general nonsense it suffices to prove that $\widetilde{\mathrm{MIC}}((X,Y,\mathfrak{P})/K)$ is a Grothendieck category. But now simply note that $\widetilde{\mathrm{MIC}}((X,Y,\mathfrak{P})/K)$ is the category of sheaves of left modules over a suitable ring of differential operators on $]Y[_\mathfrak{P}$, and the claim follows.
\end{proof}

Define functors
\begin{align*} H^i_\mathrm{dR}(]Y[_\mathfrak{P},-) : F^w\text{-}\widetilde{\mathrm{MIC}}((X,Y,\mathfrak{P})/K) &\rightarrow \widetilde{\underline{\mathbf{M}\Phi}}^{w,\nabla}_{\ekd} \\
\widetilde{\mathrm{MIC}}((X,Y,\mathfrak{P})/K) &\rightarrow \widetilde{\underline{\mathbf{M}}}^{\nabla}_{\ekd} \end{align*}
exactly as for overconvergent isocrystals, these collectively form a $\delta$-functor. Alternatively, since the categories $(F^w\text{-})\mathrm{MIC}((X,Y,\mathfrak{P})/K)$ have enough injectives, and $H^0_\mathrm{dR}(]Y[_\mathfrak{P},-)$ is left exact, let
\begin{align*} \mathbf{R}^iH^0_\mathrm{dR}(]Y[_\mathfrak{P},-) : F^w\text{-}\widetilde{\mathrm{MIC}}((X,Y,\mathfrak{P})/K) &\rightarrow \widetilde{\underline{\mathbf{M}\Phi}}^{w,\nabla}_{\ekd} \\
\widetilde{\mathrm{MIC}}((X,Y,\mathfrak{P})/K) &\rightarrow \widetilde{\underline{\mathbf{M}}}^{\nabla}_{\ekd}. \end{align*}
denote the corresponding derived functors.

\begin{proposition} There is a natural isomorphism $\mathbf{R}^iH^0_\mathrm{dR}(]Y[_\mathfrak{P},-)\cong H^i_\mathrm{dR}(]Y[_\mathfrak{P},-)$ for all $i$.
\end{proposition}

\begin{proof} Since $\mathbf{R}^iH^0_\mathrm{dR}(]Y[_\mathfrak{P},-)$ is a universal $\delta$-functor, it follows that there exists a canonical morphism $\mathbf{R}^iH^0_\mathrm{dR}(]Y[_\mathfrak{P},-)\rightarrow H^i_\mathrm{dR}(]Y[_\mathfrak{P},-)$ of $\delta$-functors. By definition, $H^i_\mathrm{dR}(]Y[_\mathfrak{P},-)$ commutes with forgetting Frobenius structures, Lemma \ref{injs} shows that $\mathbf{R}^iH^0_\mathrm{dR}(]Y[_\mathfrak{P},-)$ does as well, it therefore suffices to treat the case without Frobenius structures, and indeed to show that  $H^i_\mathrm{dR}(]Y[_\mathfrak{P},-)$ is effaceable for $i\geq1$.

Consider the category $\widetilde{\mathrm{MIC}}((X,Y,\mathfrak{P})/\ekd)$ of $j_X^\dagger\cur{O}_{]Y[_\mathfrak{P}}$-modules with integrable connection relative to $\ekd$, I first claim that the forgetful functor
\[ \widetilde{\mathrm{MIC}}((X,Y,\mathfrak{P})/K) \rightarrow \widetilde{\mathrm{MIC}}((X,Y,\mathfrak{P})/\ekd) \]
admits an exact left adjoint, which can be seen as follows. Let $\cur{D}_{]Y[_\mathfrak{P}/K}$ and $\cur{D}_{]Y[_\mathfrak{P}/S_K}$ denote the rings of differential operators on$]Y[_\mathfrak{P}$ relative to $K$ and $\ekd$ respectively, so that the categories $\widetilde{\mathrm{MIC}}((X,Y,\mathfrak{P})/K)$ and $\widetilde{\mathrm{MIC}}((X,Y,\mathfrak{P})/\ekd)$ are equivalent to the categories of left modules over $j_X^\dagger \cur{D}_{]Y[_\mathfrak{P}/K}$ and  $j_X^\dagger \cur{D}_{]Y[_\mathfrak{P}/S_K}$ respectively. There is a canonical ring homomorphism $j_X^\dagger \cur{D}_{]Y[_\mathfrak{P}/S_K} \rightarrow  j_X^\dagger \cur{D}_{]Y[_\mathfrak{P}/K}$ and the forgetful functor is simply restriction along this ring homomorphism. Hence the tensor product $j_X^\dagger \cur{D}_{]Y[_\mathfrak{P}/K} \otimes_{j_X^\dagger \cur{D}_{]Y[_\mathfrak{P}/S_K}}-$ (using the fact that $j_X^\dagger \cur{D}_{]Y[_\mathfrak{P}/K}$ may be considered as a right $j_X^\dagger \cur{D}_{]Y[_\mathfrak{P}/S_K}$-module) is the required left adjoint, using a local description shows that it is exact.

Hence the forgetful functor preserves injectives, and it therefore suffices to show that injective objects in $\widetilde{\mathrm{MIC}}((X,Y,\mathfrak{P})/\ekd)$ are acyclic for the higher de Rham cohomology functors $H^i_\mathrm{dR}(]Y[_\mathfrak{P},-)$, taking values in the category of $\ekd$-vector spaces. But this simply follows from the existence of the Spencer resolution (see for example \cite[Lemma 1.5.27]{HTT08}), which shows that $H^i_\mathrm{dR}(]Y[_\mathfrak{P},-)$ simply computes $\mathbf{R}^i\mathrm{Hom}(j_X^\dagger\cur{O}_{]Y[_\mathfrak{P}},-)$ inside the category of left $j_X^\dagger\cur{D}_{]Y[_\mathfrak{P}/S_K}$-modules.
\end{proof}

The total derived functor
\[ \mathbf{R}\Gamma_\mathrm{dR}(]Y[_\mathfrak{P},-):D^+(F^w\text{-}\widetilde{\mathrm{MIC}}^\dagger((X,Y,\mathfrak{P})/K)) \rightarrow D^+(\widetilde{\underline{\mathbf{M}\Phi}}^{w,\nabla}_{\ekd})\]
applied to $E\in \fisoc{X/K}$ therefore provides a `lift' to a suitable derived category of the cohomology groups $H^i_\rig(X/\ekd,E)$. Let 
\[ D^b(\pnm{\ekd}) \subset D^+(\widetilde{\underline{\mathbf{M}\Phi}}^{w,\nabla}_{\ekd})\]
denote the full subcategory consisting of objects whose cohomology groups are bounded, finite dimensional, and with bijective Frobenius structures. Note that I am not claiming that this is literally the bounded derived category of $\pnm{\ekd}$, but simply using this as convenient notation.

\begin{lemma} For $E\in\fisoc{X/K}$ the complex 
\[ \mathbf{R}\Gamma_\rig(X/\ekd,E):= \mathbf{R}\Gamma_\mathrm{dR}(]Y[_\mathfrak{P},E)\in D^b(\pnm{\ekd})\]
is independent of the choice of frame $(X,Y,\mathfrak{P})$.
\end{lemma}

\begin{proof} This is straightforward, any two choices of frame are dominated by a third, hence I may assume that there is a map $(X',Y',\mathfrak{P}')\rightarrow (X,Y,\mathfrak{P})$. In this case there is a morphism
\[ \mathbf{R}\Gamma_\mathrm{dR}(]Y[_\mathfrak{P},E)\rightarrow \mathbf{R}\Gamma_\mathrm{dR}(]Y'[_{\mathfrak{P}'},E')\]
and to prove it is a quasi-isomorphism it suffices first to take cohomology, and then forget all the extra structures. In other words the claim reduces to independence of $\ekd$-valued rigid cohomology of the frame.
\end{proof}

 Of course, it can be shown similarly that $\mathbf{R}\Gamma_\rig(X/\ekd,E)$ is functorial in $X$. Absolute Frobenius cohomology is then simply defined to be $\mathbf{R}\mathrm{Hom}$ inside the category $D^+(\widetilde{\underline{\mathbf{M}\Phi}}^{w,\nabla}_{\ekd})$.
 
 \begin{definition} Let $X$ be a $F$-variety and $E\in \fisoc{X/K}$. I define the absolute Frobenius cohomology groups
 \[ H^i_F(X/K,E):=H^i(\mathbf{R}\mathrm{Hom}_{D^+(\widetilde{\underline{\mathbf{M}\Phi}}^{w,\nabla}_{\ekd})}(\ekd,\mathbf{R}\Gamma_\rig(X/\ekd,E))),\]
 these are vector spaces over $K^{\sigma}$. I also define the absolute Frobenius complex
 \[ \mathbf{R}\Gamma_F(X/K,E):= \mathbf{R}\mathrm{Hom}_{D^+(\widetilde{\underline{\mathbf{M}\Phi}}^{w,\nabla}_{\ekd})}(\ekd,\mathbf{R}\Gamma_\rig(X/\ekd,E)) \in D^b(K^{\sigma=\mathrm{id}})\]
 in the bounded derived category of $K^{\sigma}$-vector spaces (we will see that these complexes are indeed bounded shortly). These are independent of the choice of frame and functorial in $X$.
 \end{definition}

When $X=\spec{F}$ and $E$ corresponds to some $\pn$-module $M$, I will generally write $H^i_F(M)$ instead of $H^i_F(X/K,E)$. These cohomology groups can be calculated as follows. For any $M\in \widetilde{\underline{\mathbf{M}\Phi}}^{w,\nabla}_{\ekd}$, let $\varphi:M\rightarrow M$ denote the $\sigma$-linear map defining the weak Frobenius structure on $M$, $\nabla:M\rightarrow M$ the connection and $\tilde{\varphi}:M\rightarrow M$ the map defined by $m\mapsto \partial_t(\sigma(t))\varphi(m)$. There is therefore a commutative diagram
\[ \xymatrix{ M \ar[r]^\nabla \ar[d]_\varphi &  M \ar[d]^{\tilde\varphi} \\ M\ar[r]^\nabla & M }\]
and an associated complex
\[ \mathcal{C}_M: 0 \rightarrow M \overset{(\nabla,\varphi-1)}{\longrightarrow} M\oplus M \overset{(1-\tilde\varphi,\nabla)}{\longrightarrow} M\rightarrow 0. \]

\begin{proposition} For any $M\in \widetilde{\underline{\mathbf{M}\Phi}}^{w,\nabla}_{\ekd}$ there is a quasi-isomorphism $\mathcal{C}_M \cong \mathbf{R}\mathrm{Hom}(\ekd,M)$ inside $D^+(K^{\sigma})$. 
\end{proposition}

\begin{proof} Consider the non-commutative ring $R:=\ekd[\partial,\varphi]$ subject to the commutation laws 
\begin{align*} \partial\lambda &=\lambda\partial+\partial_t(\lambda)\\
\varphi\lambda&=\sigma(\lambda)\varphi \\
\partial\varphi &= \partial_t(\sigma(t))\varphi\partial
\end{align*}
so that $\widetilde{\underline{\mathbf{M}\Phi}}^{w,\nabla}_{\ekd}$ is equivalent to the category of left $R$-modules. Consider the complex
\[ \mathcal{C}^*_{\ekd} :0 \rightarrow R\overset{f}{\rightarrow} R\oplus R\overset{g}{\rightarrow} R \rightarrow 0 \]
where $f(r)=(r(\tilde\varphi-1),r\partial)$ and $g(r,s)=r\partial-s(1-\varphi)$. This is easily check to be a projective resolution of $\ekd$ as an $R$-module. The complex $\mathcal{C}_M$ is then simply given by $\mathrm{Hom}(\mathcal{C}^*_{\ekd},M)$.
\end{proof}

For $M\in \pnm{\ekd}$ write 
\begin{align*} H^0_\mathrm{dR}(M)&=\ker(\nabla_M) \\
H^1_\mathrm{dR}(M)&=\mathrm{coker}(\nabla_M), 
\end{align*}
and for a (not necessarily finite dimensional) $K$-vector space $V$ with a (not necessarily bijective) Frobenius $\varphi$ write
\begin{align*} V^{\varphi=\mathrm{id}}&=\ker(\varphi-\mathrm{id}|V) \\
 V_{\varphi=\mathrm{id}}&=\mathrm{coker}(\varphi-\mathrm{id}|V).
\end{align*}
Then:
\begin{itemize} \item $H^0_F(M) \cong H^0_\mathrm{dR}(M)^{\varphi=\mathrm{id}}$;
\item there is an exact sequence
\[ 0\rightarrow H^0_\mathrm{dR}(M)_{\varphi=\mathrm{id}} \rightarrow H^1_F(M)\rightarrow H^1_\mathrm{dR}(M)^{\varphi=\mathrm{id}}\rightarrow 0;\]
\item $H^2_F(M)\cong H^1_\mathrm{dR}(M)_{\varphi=\mathrm{id}}$;
\item $H^i_F(M)=0$ for $i>2$.
\end{itemize}
It follows from Proposition \ref{exts} together with some general nonsense about derived functors that there exist canonical isomorphisms
\[ H^1_F(X/K,E_1^\vee\otimes E_2)\cong \mathrm{Ext}^1_{F^w\text{-}\widetilde{\mathrm{MIC}}((X,Y,\mathfrak{P})/K)}(E_1,E_2) \cong \mathrm{Ext}^1_{\fisoc{X/K}}(E_1,E_2) \]
for all $E_1,E_2\in \fisoc{X/K}$. Finally, the Leray spectral sequence required is the following.

\begin{proposition}\label{leray} Let $E\in \fisoc{X/K}$. Then there exists a spectral sequence
\[ E_2^{p,q}=H^p_F(H^q_\rig(X/\ekd,E))\Rightarrow H^{p+q}_F(X/K,E).\]
In particular there exists an exact sequence 
\[ 0\rightarrow H^1_F(H^0_\rig(X/\ekd,E)) \rightarrow H^1_F(X/K,E)\rightarrow H^0_F(H^1_\rig(X/\ekd,E))\rightarrow  H^2_F(H^0_\rig(X/\ekd,E)) \rightarrow H^2_F(X/K,E) \]
of low degree terms.
\end{proposition}

\begin{proof} This is simply the Grothendieck spectral sequence associated to the composition
\[D^+(F^w\text{-}\widetilde{\mathrm{MIC}}^\dagger((X,Y,\mathfrak{P})/\ekd))\rightarrow D^+(\widetilde{\underline{\mathbf{M}\Phi}}^{w,\nabla}_{\ekd})\rightarrow D^+(K^\sigma).\]
\end{proof}

\begin{remark} \label{globalabs} Actually, I will also need a version of absolute Frobenius cohomology in the global case, as well as the local case, that is for algebraic varieties over $k$, rather than $F$. Needless to say, this is entirely similar to (and, if anything, easier than) the local case handled above. I will therefore refer to such absolute Frobenius cohomology groups $H^i_F(X/K,E)$ whenever $X/k$ and $E\in \fisoc{X/K}$ without further mention. Essentially all the expected formalism holds, in particular there will be a Leray spectral sequence as in Proposition \ref{leray} above.
\end{remark}

\section{The fundamental group of a Tannakian category}\label{fgtc}

In this section I will briefly recap some of the material from \cite{Del89} on algebraic geometry in Tannakian categories, and in particular the notion of the fundamental group of a Tannakian category. So let $\mathcal{T}$ be a Tannakian category over an arbitrary field $k$, with unit object $\mathbf{1}$. I have already essentially described (Definition \ref{hopfenriched}) what should constitute a Hopf algebra in $\mathcal{T}$, this can be put into a slightly broader framework as follows.

\begin{definition} A commutative $\mathcal{T}$-algebra is an object $A\in \mathrm{Ind}(\mathcal{T})$ together with maps $m:A\otimes A\rightarrow A$ and $u:\mathbf{1}\rightarrow A$ such that the diagrams
\[ \xymatrix{ A\otimes A \ar[r]^-{\tau}\ar[dr]_-{m}&  A\otimes A \ar[d]^-{m}& & A\otimes A \otimes A \ar[r]^-{m\otimes\mathrm{id}}\ar[d]_-{\mathrm{id}\otimes m}& A\otimes A \ar[d]^-{m}&  & A\ar[d]_-{u\otimes \mathrm{id}} \ar[dr]^-{\mathrm{id}} & &  \\
 & A  & & A\otimes A \ar[r]^-{m} & A &  & A\otimes A \ar[r]^-{m}& A} \]
commute, where $\tau$ is the `switch' map. A morphism of $\mathcal{T}$-algebras is then a morphism commuting with $m$ and $u$. The category of commutative $\mathcal{T}$-algebras is denoted $\mathrm{Alg}_\mathcal{T}$. 
\end{definition}

Define the category of affine $\mathcal{T}$-schemes to be the opposite category $\mathrm{Aff}_\mathcal{T}:=\mathrm{Alg}_\mathcal{T}^\mathrm{op}$.

\begin{definition} An affine group scheme over $\mathcal{T}$ is a group object in $\mathrm{Aff}_\mathcal{T}$. 
\end{definition}

Of course this can be described in terms of `Hopf $\mathcal{T}$-algebras' as in Definition \ref{hopfenriched}. If $L$ is a $k$-field and $\omega:\mathcal{T}\rightarrow \mathrm{Vec}_L$ is a fibre functor, then $\omega$ induces functors
\begin{align*} \omega:\mathrm{Alg}_\mathcal{T}&\rightarrow \mathrm{Alg}_L \\
\omega:\mathrm{Aff}_\mathcal{T} &\rightarrow \mathrm{Aff}_L \\
\omega:\mathrm{AffGrp}_\mathcal{T} & \rightarrow \mathrm{AffGrp}_L
\end{align*}
from the categories of $\mathcal{T}$-algebras (resp. affine $\mathcal{T}$-schemes, affine group schemes over $\mathcal{T}$) to the category of $L$-algebras (resp. affine $L$-schemes, affine group schemes over $L$).

For any such fibre functor $\omega$ there is an affine group scheme $\pi(\mathcal{T},\omega)$ over $L$ representing the functor of automorphisms of $\omega$, and the basic idea of the fundamental group of a Tannakian category is to package these together into a single affine group scheme over $\mathcal{T}$.

\begin{proposition}[\cite{Del89}, 5.11 and 6.1] Let $\mathcal{T}$ be a Tannakian category. Then there exists a unique affine group scheme $\pi(\mathcal{T})$ over $\mathcal{T}$ such that for any fibre functor
\[ \omega:\mathcal{T}\rightarrow \mathrm{Vec}_L \]
there is an isomorphism
\[  \omega(\pi(\mathcal{T}))\cong \pi(\mathcal{T},\omega),\]
functorial in $\omega$. The affine group scheme $\pi(\mathcal{T})$ is called the fundamental group of $\mathcal{T}$.
\end{proposition}

This is of course functorial in $\mathcal{T}$, in that whenever there is a (faithful, exact, tensor) functor $\eta:\mathcal{T}\rightarrow \mathcal{T}'$ of Tannakian categories there is an induced homomorphism
\[ \eta^*:\pi(\mathcal{T}')\rightarrow \eta(\pi(\mathcal{T}))\]
of fundamental groups. To put this in a relative situation, suppose that $\cur{T},\cur{T}'$ is a pair of Tannakian categories and there is a pair of (faithful, exact, tensor) functors $\eta:\cur{T}\leftrightarrows \cur{T}':\eta'$ such that $\eta \circ\eta' \cong \mathrm{id}$. Then there is a homomorphism
\[  \eta'^*: \pi(\mathcal{T}) \rightarrow \eta'(\pi(\mathcal{T}'))\]
of affine group schemes over $\cur{T}'$, define 
\[ \pi(\cur{T}/\cur{T}',\eta):= \eta( \ker \eta'^*);\]
this is an affine group scheme over $\cur{T}'$. I will in general refer to $\pi(\cur{T}/\cur{T}',\eta)$ as the relative fundamental group associated to the pair of functors $\eta:\cur{T}\leftrightarrows \cur{T}':\eta'$, note the asymmetry in the requirement that $\eta \circ\eta' \cong \mathrm{id}$. Finally, I will need  the following result.

\begin{proposition}[\cite{Str07}, Theorem 26.4] \label{se} Let $\mathcal{T}$ be a Tannakian category over a field $k$, let $L/k$ be an arbitrary field extension, and let $\omega:\mathcal{T}\rightarrow \mathrm{Vec}_L$ be a fibre functor with values in $L$. Then there is a canonical equivalence of categories
\[ \mathcal{T}_L\isomto \mathrm{Rep}_{L}(\pi(\mathcal{T},\omega)) \]
between the extension of scalars of $\mathcal{T}$ to $L$ and the category of $L$-linear representations of $\pi(\mathcal{T},\omega)$.
\end{proposition}

\section{\texorpdfstring{$\pn$}{pn}-module structure via the Tannakian formalism}\label{maini}

Let $X/F$ be a geometrically connected variety, and $x\in X(F)$ a rational point. Then there is a fibre functor
$$ x^*:\cur{N}\mathrm{Isoc}^\dagger(X/\ekd)\rightarrow \mathrm{Vec}_{\ekd}
$$
on the Tannakian category $\cur{N}\mathrm{Isoc}^\dagger(X/\ekd)$ and the (unipotent) rigid fundamental group of $X$ with base point $x$ is defined by
$$ \pi_1^\rig(X/\ekd,x):= \underline{\mathrm{Aut}}^\otimes(x^*) $$
using the Tannakian formalism. In this section I will show how to put a `$\pn$-module structure' on $\pi_1^\rig(X/\ekd,x)$, using similar methods to those in \cite{Laz15}. The exact meaning of a $\pn$-module structure in the non-abelian situation is the following.

\begin{definition} A $\pn$-module structure on $\pi_1^\rig(X/\ekd,x)$ is a Hopf $\underline{\mathbf{M}\Phi}_{\ekd}^\nabla$-algebra $\cur{A}^{\rig,\vee}_{\infty,x}$ whose underlying Hopf $\ekd$-algebra is $A^{\rig,\vee}_{\infty,x}:=\mathcal{O}(\pi_1^\rig(X/\ekd,x))$.
\end{definition} 
To define such a structure, consider the category $F\text{-}\mathrm{Isoc}^\dagger(X/K)$ of $F$-isocrystals on $X/K$, and the natural `pullback' functor
$$ -\otimes \mathcal{O}_{X/K}^\dagger: \underline{\mathbf{M}\Phi}_{\ekd}^{\nabla} \rightarrow F\text{-}\mathrm{Isoc}^\dagger(X/K).
$$Denote by $\cur{N}_\mathrm{rel}F\text{-}\mathrm{Isoc}^\dagger(X/K)$ the full subcategory of $F\text{-}\mathrm{Isoc}^\dagger(X/K)$ consisting of \emph{relatively unipotent} objects, that is objects $E$ admitting a filtration $F_\bu E$ whose associated graded $\bigoplus _i F_{i+1}E/F_iE$ is in the essential image of $-\otimes \mathcal{O}_{X/K}^\dagger$. This should not be confused with the subcategory $\cur{N}\fisoc{X/K}$ consisting of objects which are iterated extensions of the unit object $\mathcal{O}_{X/K}^\dagger$.

\begin{lemma} The category $\cur{N}_\mathrm{rel}\fisoc{X/K}$ is Tannakian over $K^{\sigma}$.
\end{lemma}

\begin{proof} Follows from Corollary \ref{tann2}.
\end{proof}

The point $x$ gives rise to a functor 
$$ x^*: \cur{N}_\mathrm{rel}F\text{-}\mathrm{Isoc}^\dagger(X/K)\rightarrow \pnm{\ekd}
$$
which splits $-\otimes \mathcal{O}_{X/K}^\dagger$, and as described in \S\ref{fgtc} there is associated to this set-up an `affine group scheme over $\pnm{\ekd}$', which I will for now denote by $G^\rig_{X,x}$. This is obtained by pulling back the kernel of the natural map of fundamental groups
$$ \pi(\cur{N}_\mathrm{rel}F\text{-}\mathrm{Isoc}^\dagger(X/K)) \rightarrow \pi(\pnm{\ekd}))\otimes \mathcal{O}_{X/K}^\dagger
$$
by $x^*$. Concretely, this means that $G^\rig_{X,x}$ is the `formal spectrum' of an Hopf algebra $\cur{A}^{\rig,\vee}_{\infty,x}$ in $\pnm{\ekd}$, and if $\omega$ is a fibre functor on $\pnm{\ekd}$, with values in some field $L$, then $\omega(\cur{A}^{\rig,\vee}_{\infty,x} )$ is a Hopf algebra over $L$, functorially in $\omega$; I will denote by $\omega(G^\rig_{X,x})$ the spectrum of this Hopf algebra. In particular, considering the natural fibre functor
$$ \omega: \pnm{\ekd} \rightarrow \mathrm{Vec}_{\ekd}
$$
given by forgetting the connection and Frobenius produces an affine group scheme $\omega(G^\rig_{X,x})$ over $\ekd$. The forgetful functor
$$ \cur{N}_\mathrm{rel}F\text{-}\mathrm{Isoc}^\dagger(X/K) \rightarrow \cur{N}\mathrm{Isoc}^\dagger(X/\ekd)
$$
induces a homomorphism of group schemes
$$ \pi_1^\rig(X/\ekd,x)\rightarrow \omega(G^\rig_{X,x})
$$
over $\ekd$.

\begin{theorem} \label{bc}This map is an isomorphism.
\end{theorem}

\begin{proof} Let $\cur{N}_\mathrm{rel}F\text{-}\mathrm{Isoc}^\dagger(X/K)\otimes {\ekd}$ and $\pnm{\ekd}\otimes \ekd$ denote the formal extension of scalars of from $K^\sigma$ to $\ekd$. Note that by  of \cite[Theorem 24.1]{Str07} the natural forgetful functor
\[ \cur{N}_\mathrm{rel}F\text{-}\mathrm{Isoc}^\dagger(X/K)\rightarrow \cur{N}\mathrm{Isoc}^\dagger(X/\ekd)\]
factors uniquely through $\cur{N}_\mathrm{rel}F\text{-}\mathrm{Isoc}^\dagger(X/K)\otimes {\ekd}$. Then using Lemma \ref{se} together with the main result from \cite[Appendix A]{EHS08}, it suffices to prove the following three claims:
\begin{enumerate} \item if $E \in \cur{N}_\mathrm{rel}F\text{-}\mathrm{Isoc}^\dagger(X/K)\otimes {\ekd}$ is such that the underlying object in $\cur{N}\mathrm{Isoc}^\dagger(X/\ekd)$ is constant, then $E\cong M\otimes\mathcal{O}_{X/K}^\dagger$ for some $M\in \pnm{\ekd} \otimes \ekd $;
\item if $E \in  \cur{N}_\mathrm{rel}F\text{-}\mathrm{Isoc}^\dagger(X/K) \otimes \ekd $ and $E'_0\subset E$  in $\cur{N}\mathrm{Isoc}^\dagger(X/\ekd)$ is the largest constant sub-isocrystal over $\ekd$, then there exists some $E'\subset E$ in $\cur{N}_\mathrm{rel}F\text{-}\mathrm{Isoc}(X/K) \otimes \ekd$ whose underlying object in $\cur{N}\mathrm{Isoc}^\dagger(X/\ekd)$ is $E'_0$;
\item any $E \in\cur{N}\mathrm{Isoc}^\dagger(X/\ekd)$ is a quotient of an object in $\cur{N}_\mathrm{rel}F\text{-}\mathrm{Isoc}^\dagger(X/K)\otimes \ekd$.
\end{enumerate}

The first two are entirely straightforward, I will show them for objects in the non-scalar extended category $\cur{N}_\mathrm{rel}F\text{-}\mathrm{Isoc}^\dagger(X/K)$, the claim for the more general objects follows more or less instantly from the functoriality of the constructions. The point is that for any $E\in F\text{-}\mathrm{Isoc}^\dagger(X/K)$ there is a natural adjunction map
$$ H^0_\rig(X/\ekd,E) \otimes \mathcal{O}_{X/K}^\dagger \rightarrow E
$$
in $F\text{-}\mathrm{Isoc}^\dagger(X/K)$, and $E$ is trivial in $\mathrm{Isoc}^\dagger(X/\ekd)$ iff this is an isomorphism in $\mathrm{Isoc}^\dagger(X/\ekd)$, which is iff it is an isomorphism in $F\text{-}\mathrm{Isoc}^\dagger(X/K)$. Similarly, if $E'_0\subset E$ is the largest constant sub-isocrystal over $\ekd$, then $E':= H^0_\rig(X/\ekd,E)\otimes \mathcal{O}_{X/K}^\dagger$ is a subobject of $E$ giving rise to $E'_0$. 

The third, however, is a lot more involved, and involves constructing certain `universal' unoipotent objects $U_n\in \isoc{X/\ekd}$, of which every other unipotent object is a quotient, and then showing that these universal objects extend to $\cur{N}_\mathrm{rel}\fisoc{X/K}$, and therefore (trivially) to $\cur{N}_\mathrm{rel}\fisoc{X/K}\otimes \ekd$. The argument is essentially the same as that in \cite[\S3]{Laz15}, following the original idea of Hadian \cite[\S2]{Had11}. I will therefore only describe how to construct these objects and state their properties, for proofs the reader should consult \cite{Had11,Laz15}.

Recursively define objects $U_n\in\isoc{X/\ekd}$ such that $U_1=\mathcal{O}^\dagger_{X/\ekd}$ is the trivial object, and $U_{n+1}$ is an extension of $U_n$ by $H^1_\rig(X/\ekd,U_n^\vee)^\vee \otimes \mathcal{O}_{X/K}^\dagger$ corresponding to the identity under the composite isomorphism
\begin{align*} \mathrm{Ext}_{\isoc{X/\ekd}}&(U_n,H^1_\rig(X/\ekd,U_n^\vee)^\vee\otimes \mathcal{O}_{X/K}^\dagger) \\ &\cong H^1_\rig(X/\ekd,U_n^\vee\otimes H^1_\rig(X/\ekd,U_n^\vee)^\vee) \\
&\cong H^1_\rig(X/\ekd,U_n^\vee)\otimes_{\ekd} H^1_\rig(X/\ekd,U_n^\vee)^\vee \\
&\cong \mathrm{End}_{\ekd}(H^1_\rig(X/\ekd,U_n^\vee)).
\end{align*}
Letting $u_1=1\in (\mathcal{O}_{X/\ekd}^\dagger)_x$ and choosing a compatible system $u_n\in (U_n)_x $ mapping to $u_1$, then the pro-object $\left\{(U_n,u_n)\right\}$ is \emph{universal} in the sense that for any object $E\in\cur{N}\isoc{X/\ekd}$ of unipotence degree $\leq n$, and any $e\in E_x$, there exists a unique morphism $\phi:U_n\rightarrow E $ such that $\phi(u_n)=e$. Hence it suffices to extend the $U_n$ to objects $W_n\in \cur{N}_\mathrm{rel}F\text{-}\isoc{X/K}$. 

As in \cite{Laz15}, I will do this inductively, starting with $W_1=\mathcal{O}^\dagger_{X/K}$ extending $U_1=\mathcal{O}^\dagger_{X/\ekd}$. I will also strengthen the induction hypothesis to require in addition that:
\begin{enumerate} \item $H^0_\rig(X/\ekd,W_{n}^\vee)\cong \ekd$ as $\pn$-modules;
\item there exists a map $x^*W_n^\vee\rightarrow \ekd$ of $\pn$-modules such that the composite 
\[ \ekd\cong H^0_\rig(X/\ekd,W_{n}^\vee) \cong x^*(H^0_\rig(X/\ekd,W_{n}^\vee)\otimes \mathcal{O}_{X/K}^\dagger) \rightarrow x^*W_{n}^\vee \rightarrow \ekd \]
is an isomorphism.
\end{enumerate}
Assume that $W_n$ has been constructed satisfying all the required conditions, I will construct $W_{n+1}$ as an extension of $W_n$ by $H^1_\rig(X/\ekd,W_n^\vee)^\vee \otimes\mathcal{O}_{X/K}^\dagger$. Using the five term exact sequence arising from Proposition \ref{leray}, together with the `projection formula' (i.e. Proposition \ref{proj}) and the induction hypothesis there is the following exact sequence
\begin{align*} 0 \rightarrow H^1_F(H^1_\rig(X/\ekd,W_n^\vee)^\vee) \rightarrow \mathrm{Ext}^1_{\fisoc{X/K}}&(W_n,H^1_\rig(X/\ekd,W_n^\vee)^\vee \otimes\mathcal{O}_{X/K}^\dagger) \\
 \rightarrow \mathrm{End}_{\pnm{\ekd}}(H^1_\rig(X/\ekd,W_n^\vee)) &\rightarrow H^2_F(H^1_\rig(X/\ekd,W_n^\vee)^\vee) \\
 \rightarrow H^2_F(X/K,&W_n^\vee\otimes H^1_\rig(X/\ekd,W_n^\vee)^\vee)\rightarrow \ldots
 \end{align*}
Now use the hypothesised map $x^*W_n^\vee\rightarrow \ekd$ to split the maps
\begin{align*} H^1_F(H^1_\rig(X/\ekd,W_n^\vee)^\vee) \rightarrow \mathrm{Ext}^1_{\fisoc{X/K}}&(W_n,H^1_\rig(X/\ekd,W_n^\vee)^\vee \otimes\mathcal{O}_{X/K}^\dagger) \\ &=H^1_F(X/K,W_n^\vee\otimes H^1_\rig(X/\ekd,W_n^\vee)^\vee) 
\end{align*}
and 
\[
H^2_F(H^1_\rig(X/\ekd,W_n^\vee)^\vee) \rightarrow H^2_F(X/K,W_n^\vee\otimes H^1_\rig(X/\ekd,W_n^\vee)^\vee)
\]
appearing in the above sequence, and hence define $W_{n+1}$ to be the extension class 
\[[W_{n+1}]\in \mathrm{Ext}^1_{\fisoc{X/K}}(W_n,H^1_\rig(X/\ekd,W_n^\vee)^\vee \otimes\mathcal{O}_{X/K}^\dagger)\]
which maps to the identity in $\mathrm{End}_{\pnm{\ekd}}(H^1_\rig(X/\ekd,W_n^\vee))$ and to zero in $H^1_F(H^1_\rig(X/\ekd,W_n^\vee)^\vee)$. This is clearly an extension of $U_{n+1}$ to $\mathcal{N}_f\fisoc{X/K}$. To complete the proof of the theorem, therefore, I must show that $H^0_\rig(X/\ekd,W_{n+1}^\vee)\cong \ekd$ and that there exists a map $x^*W_{n+1}^\vee\rightarrow \ekd$ such that 
\[ \ekd\cong H^0_\rig(X/\ekd,W_{n}^\vee) \cong x^*(H^0_\rig(X/\ekd,W_{n}^\vee)\otimes \mathcal{O}_{X/K}^\dagger) \rightarrow x^*W_{n}^\vee \rightarrow \ekd \]
is an isomorphism.

To show the first of these, consider the exact sequence in cohomology
\[ 0\rightarrow H^0_\rig(X/\ekd,W_n^\vee)\rightarrow H^0_\rig(X/\ekd,W_{n+1}^\vee)\rightarrow \ldots \]
associated to the short exact sequence
\[ 0\rightarrow W_n^\vee \rightarrow W_{n+1}^\vee \rightarrow H^1_\rig(X/\ekd,W_n^\vee)\otimes\mathcal{O}_{X/K}^\dagger\rightarrow 0.\]
in $\fisoc{X/K}$. The induced map
\[H^0_\rig(X/\ekd,W_n^\vee)\rightarrow H^0_\rig(X/\ekd,W_{n+1}^\vee) \]
is an isomorphism after forgetting $\pn$-module structures (c.f. Lemma 3.17 of \cite{Laz15} and the following paragraph), and hence is an isomorphism. Therefore by the induction hypothesis $H^0_\rig(X/\ekd,W_{n+1}^\vee)\cong \ekd$ as required.

To show the second, note that by construction the exact sequence
\[ 0\rightarrow x^*W_n^\vee \rightarrow x^*W_{n+1}^\vee \rightarrow H^1_\rig(X/\ekd,W_n^\vee) \rightarrow 0\]
splits after pushing out via $x^*W_n^\vee\rightarrow \ekd$. This splitting induces a map $x^*W_{n+1}^\vee\rightarrow \ekd$ which by an easy diagram chase is seen to satisfy the required conditions.
\end{proof}

The following then summarises the main result of this section.

\begin{corollary} There is a natural isomorphism $\omega(\cur{A}^{\rig,\vee}_{\infty,x})\cong A^{\rig,\vee}_{\infty,x}$ of Hopf algebras over $\ekd$, inducing the structure of a non-abelian $\pn$-module on $\pi_1^\rig(X/\ekd,x)$.
\end{corollary}

No confusion will therefore arise if I denote $G^\rig_{X,x}$ by $\pi_1^\rig(X/\ekd,x)$ from now on, i.e. $\pi_1^\rig(X/\ekd,x)$ will be considered with its canonical $\pn$-module structure. At this point perhaps a few sanity checks are in order, the first two are relatively straightforward.

\begin{proposition} Let $X/F$ be geometrically connected, and $x\in X(F)$. There is a natural isomorphism
\[ \pi_1^\rig(X/\ekd,x) \otimes_{\ekd} \ek \cong \pi_1^\rig(X/\ek,x) \]
of non-abelian $\pn$-modules 
\end{proposition}

\begin{remark} Actually, I haven't described the $\pn$-module structure on $\pi_1^\rig(X/\ek,x)$, but it is constructed in essentially the same way as that on $\pi_1^\rig(X/\ekd,x)$. 
\end{remark}

\begin{proof} There is a naturally defined base change map
\[  \pi_1^\rig(X/\ek,x)\rightarrow \pi_1^\rig(X/\ekd,x) \otimes_{\ekd} \ek \] 
(coming from the quasi-completion functor $E\mapsto \hat{E}$ on isocrystals) and to show it is an isomorphism it suffices to do so after forgetting the connection and Frobenius, in other words to show that the functor
\[ \mathcal{N}\isoc{X/\ekd}\otimes_{\ekd} \ek \rightarrow \mathcal{N}\isoc{X/\ek} \]
is an equivalence of categories. But this can be expressed purely in terms of cohomology: to show that it is fully faithful amounts to showing that 
\[ H^0_\rig(X/\ekd,E)\otimes_{\ekd} \ek \rightarrow H^0_\rig(X/\ek,\hat{E}) \]
is an isomorphism for all $E\in \mathcal{N}\isoc{X/\ekd}$, and to show that it is essentially surjective amounts to showing that 
\[ H^1_\rig(X/\ekd,E)\otimes_{\ekd} \ek \rightarrow H^1_\rig(X/\ek,\hat{E}) \]
is an isomorphism for all $E\in \mathcal{N}\isoc{X/\ekd}$. For $E=\mathcal{O}_{X/\ekd}^\dagger$ this is essentially the main result of \cite{LP16}, and the general case then follows from the five lemma.
\end{proof}

\begin{proposition} Let $X/F$ be geometrically connected, and $x\in X(F)$. Then there is an isomorphism
\[ \pi_1^\rig(X/\ekd,x)^\mathrm{ab} \cong H^1_\rig(X/\ekd)^\vee \]
of $\pn$-modules over $\ekd$.
\end{proposition}

\begin{remark}  Here $\pi_1^\rig(X/\ekd,x)^\mathrm{ab}$ refers to the abelianisation of $\pi_1^\rig(X/\ekd,x)$, \emph{a priori} it is a (possibly infinite dimensional) vector group scheme in the category of $\pn$-modules over $\ekd$. By  \cite[Corollary 1.1.10]{Shi00} together with Proposition \ref{exts} and the main finiteness results of \cite{LP16} it is in fact finite dimensional, and hence the vector group scheme associated to some finite dimensional, abelian $\pn$-module over $\ekd$ (i.e. an object of $\pnm{\ekd}$).
\end{remark}

\begin{proof} By \cite[Corollary 1.1.10]{Shi00} the claim holds after forgetting $\pn$-module structures. Enriching this to include $\pn$-module structures simply follows from the `enriched' version of \cite[Corollary 1.1.8]{Shi00}, which is \cite[Proposition 2.3]{Laz15}. 
\end{proof}

The last comparison I want to show will be a touch more involved, and will be the subject of the next section.

\section{Comparison with the global situation}\label{compglob}

In this section I want to compare non-abelian $\pn$-module $\pi_1^\rig(X/\ekd,x)$ with the `relative fundamental group' of a smooth and proper family constructed in \cite{Laz15}. So let $C$ be a smooth, geometrically connected curve over $k$, and $f:\mathcal{X}\rightarrow C$ a smooth and proper morphism with geometrically connected fibres. Fix a $k$-rational point $c\in C(k)$ and an isomorphism $F\cong  \widehat{k(C)}_c$ between $F$ and the completion of the function field of $C$ at $c$, let $X=\mathcal{X}\times_C F$ be the corresponding smooth, proper, geometrically connected variety over $F$. Suppose that there is given a section $s:C\rightarrow \mathcal{X}$ of $f$, and let $x\in X(F)$ be the corresponding point induced by pull-back.

In this situation, I constructed in \cite{Laz15} a `non-abelian' overconvergent $F$-isocrystal $\pi_1^\rig(\mathcal{X}/C,s)\in \fisoc{C/K}$ whose fibre at every \emph{closed} point $c'\in C$ is the unipotent rigid fundamental group of $\mathscr{X}_{c'}$. This fundamental group is constructed as the `relative' fundamental group associated to the pair of functors 
\[ s^*:\cur{N}_f\fisoc{\cur{X}/K} \leftrightarrows \fisoc{C/K} :f^*  \]
where $\cur{N}_f\fisoc{\cur{X}/K}$ is the category of relatively unipotent $F$-isocrystals on $\cur{X}$. 

In \cite[\S6]{Tsu98}, Tsuzuki constructs a functor
\[ i^*_c: \fisoc{C/K}\rightarrow \pnm{\ekd} \]
by pulling back to a punctured formal neighbourhood of the given point $c\in C(k)$, and the main result of this section is the following comparison theorem.

\begin{theorem}\label{bc2} There is a canonical isomorphism
\[ i^*_c \left(\pi_1^\rig(\mathcal{X}/C,s)\right) \cong \pi_1^\rig(X/\ekd,x)\]
of non-abelian $\pn$-modules over $\ekd$. 
\end{theorem}

\begin{remark} The proof of this theorem, while not strictly necessary for the results in this article, does provide a model for some base change results that I will prove later on. I hope, then, that the reader will forgive me for somewhat labouring the details here.
\end{remark}

The first step is to extend Tsuzuki's pullback functor to higher dimensions, in particular to produce a pullback functor 
\[ i_c^*: \fisoc{\mathcal{X}/K} \rightarrow \fisoc{X/K}.\]
This is relatively straightforward, for simplicity I will assume that there exists a smooth and proper frame $(\mathcal{X},\mathcal{Y},\mathfrak{P})$ over $\cur{V}$, so that $\isoc{\mathcal{X}/K}$ can be identified with $\mathrm{MIC}^\dagger((\mathcal{X},\mathcal{Y},\mathfrak{P})/K)$. Since $C$ is a curve, there is a smooth and proper frame $(C,\overline{C},\mathfrak{C})$ (again in the classical sense) and after possibly changing $(\mathcal{X},\mathcal{Y},\mathfrak{P})$ I may assume that there exists a smooth and proper morphism of frames
\[ (\mathcal{X},\mathcal{Y},\mathfrak{P})\rightarrow (C,\overline{C},\mathfrak{C}).\]
Now associated to the point $c$, I can construct (exactly as in \cite[\S6]{Tsu98}) a morphism of triples
\[ \left(\spec{F},\spec{R},\spf{\cur{V}\pow{t}} \right)\rightarrow (C,\overline{C},\mathfrak{C})\]
and hence a Catersian diagram
\[ \xymatrix{ (X,Y,\mathfrak{Q}) \ar[r]\ar[d] & (\mathscr{X},\mathscr{Y},\mathfrak{P}) \ar[d] \\ \left(\spec{F},\spec{R},\spf{\cur{V}\pow{t}} \right)\ar[r] &(C,\overline{C},\mathfrak{C})  } \]
where $(X,Y\mathfrak{Q})$ is a smooth and proper frame over $\cur{V}\pow{t}$. It follows that $\isoc{X/K}$ can be identified with $\mathrm{MIC}^\dagger((X,Y,\mathfrak{P})/K)$ and there is a functor $i_c^*: \isoc{\mathcal{X}/K} \rightarrow \isoc{X/K}$ which is simply the pullback functor
\[  \mathrm{MIC}^\dagger((\mathcal{X},\mathcal{Y},\mathfrak{P})/K) \rightarrow \mathrm{MIC}^\dagger((X,Y,\mathfrak{P})/K). \]
Of course, this doesn't depend on any choices and is therefore functorial, in particular it is functorial with respect to Frobenius and therefore induces a functor
\[ i_c^*: \fisoc{\mathcal{X}/K} \rightarrow \fisoc{X/K}. \]
There is therefore a commutative diagram
\[ \xymatrix{ \cur{N}_f\fisoc{\mathcal{X}/K}  \ar@/^1pc/[d]^{s^*}\ar[r]^{i_c^*} & \cur{N}_\mathrm{rel}\fisoc{X/K} \ar@/^1pc/[d]^{x^*} \\   \fisoc{C/K} \ar[r]^{i_c^*}\ar[u]^{f^*} & \pnm{\ekd} \ar[u]^{-\otimes \cur{O}_{X/K}^\dagger}}  \]
and therefore by some general nonsense a homomorphism
\[ \pi_1^\rig(X/\ekd,x)\rightarrow i_c^*\pi_1^\rig(\mathcal{X}/C,s) \]
of affine group schemes over $\pnm{\ekd}$. To check that it is an isomorphism it suffices to do so after applying a fibre functor, I will choose the forgetful functor $\pnm{\ekd}\rightarrow \mathrm{Vec}_{\ekd}$ to vector spaces over $\ekd$. Extend the above commutative diagram to the diagram
\[ \xymatrix{ \cur{N}_f\fisoc{\mathcal{X}/K} \ar[r]^{i_c^*} & \cur{N}_\mathrm{rel}\fisoc{X/K} \ar[r]  & \cur{N}\isoc{X/\ekd}\\   \fisoc{C/K} \ar[r]^{i_c^*}\ar[u]^{f^*} & \pnm{\ekd} \ar[u]_{-\otimes \cur{O}_{X/K}^\dagger}  \ar[r] & \mathrm{Vec}_{\ekd},  \ar[u]_{-\otimes \cur{O}_{X/K}^\dagger}}\]
by Theorem \ref{bc} the right hand square identifies $\pi_1^\rig(X/\ekd,x)$ (without its $\pn$-module structure) with the Tannakian fundamental group of the category $\cur{N}\isoc{X/\ekd}$, and it therefore suffices to prove the corresponding `base change' result for the square 
\[ \xymatrix{ \cur{N}_f\fisoc{\mathcal{X}/K}\ar@/^1pc/[d]^{s^*} \ar[r]^{i_c^*}  & \cur{N}_\mathrm{rel}\isoc{X/\ekd}\ar@/^1pc/[d]^{x^*} \\   \fisoc{C/K} \ar[r]^{i_c^*}\ar[u]^{f^*} &  \mathrm{Vec}_{\ekd}  \ar[u]^{-\otimes \cur{O}_{X/K}^\dagger}}\]
To spell this out then, the bottom arrow is a fibre functor on $\fisoc{C/K}$, and the square induces a homomorphism
\[ \pi_1^\rig(X/\ekd,x)\rightarrow i^*_c(\pi_1^\rig(\mathcal{X}/C,s))\]
of affine group schemes over $\ekd$ (with no $\pn$-module structure). Theorem \ref{bc2} will then follow from the next result.

\begin{theorem}\label{cgt} This induced map $\pi_1^\rig(X/\ekd,x)\rightarrow i_c^*(\pi_1^\rig(\mathcal{X}/C,s))$ is an isomorphism.
\end{theorem}

\begin{remark} The proof of this result is essentially \emph{exactly} the same as the proof of Theorem \ref{bc}, with one crucial extra detail. In that proof, the cohomology groups $H^i_\rig(X/\ekd,E)$ in fact played two subtly different roles, one with $E\in \fisoc{X/K}$, where they come with an extra $\pn$-module structure, and one with $E\in \isoc{X/\ekd}$ in which case they are simply vector spaces over $\ekd$. 

The point is that in carrying out the `same' proof here, these two roles need to be separated more clearly.  In the first case, the $\pn$-module $ H^i_\rig(X/\ekd,E)$ associated to some $E\in \fisoc{X/K}$ will be replaced by the higher direct image sheaves $\mathbf{R}^if_* E$ associated to some $E\in \fisoc{\cur{X}/K}$, that these are in fact overconvergent $F$-isocrystals on $C/K$ is proved, for example, in \cite[\S3]{Laz15}. In the second case nothing changes.

The fact that $H^i_\rig(X/\ekd,E)$ plays two different roles is essentially a form of base change, and now this becomes a non-trivial result I have to prove: that given an overconvergent $F$-isocrystal $E\in \fisoc{\cur{X}/K}$, there is a `base change' isomorphism
\[ i_c^*(\mathbf{R}^if_*E) \cong H^i_\rig(X/\ekd,i_c^*E)\]
of vector spaces over $\ekd$. This was proved in \cite{LP16} for constant coefficients, follows by the projection formula for anything pulled back from $C$, and for everything relatively unipotent by the five lemma.

 With these caveats in place, I will now proceed to copy out word for word the proof of Theorem \ref{bc}. As explained above, the reason for going through this in such tortuous detail is that it will provide the template for other base change results still to come. The slight confusion of these roles played by $H^i_\rig(X/\ekd,E)$ in the proof of Theorem \ref{bc} as discussed above means that it is not completely ideal to have this as the only fully worked out explanation of the method. Of course, if I were cleverer I would simply axiomatise the situation.
\end{remark}

\begin{proof} As before, it suffices to prove the following three claims:
\begin{enumerate} \item if $E \in \cur{N}_fF\text{-}\mathrm{Isoc}^\dagger(\cur{X}/K)$ is such that $i_c^*E$ is constant, then $E\cong f^*M$ for some $M\in \fisoc{C/K} $;
\item if $E \in  \cur{N}_fF\text{-}\mathrm{Isoc}^\dagger(\cur{X}/K) $ and $E_0'\subset i_c^*E$  in $\cur{N}\mathrm{Isoc}^\dagger(X/\ekd)$ is the largest constant sub-isocrystal over $\ekd$, then there exists some $E'\subset E$ in $\cur{N}_fF\text{-}\mathrm{Isoc}(\cur{X}/K)$ such that $i_c^*E'=E_0'$;
\item any $E' \in\cur{N}\mathrm{Isoc}^\dagger(\cur{X}/\ekd)$ is a quotient of $i_c^*E$ for some $E\in\cur{N}_fF\text{-}\mathrm{Isoc}^\dagger(\cur{X}/K)$.
\end{enumerate}

This time I use the adjunction map
\[ f^*f_*E \rightarrow E \]
for any $E\in \cur{N}_fF\text{-}\mathrm{Isoc}^\dagger(\cur{X}/K)$. The base change isomorphism
\[ i_c^*(\mathbf{R}^if_*E) \cong H^i_\rig(X/\ekd,i_c^*E)\]
shows that pulling back this adjunction map via $i_c^*$ gives the adjunction map
\[ H^0_\rig(X/\ekd,i_c^*E) \otimes \cur{O}_{X/\ekd}^\dagger\rightarrow i_c^*E. \]
associated to $i_c^*E$. Then $i_c^*E$ being constant implies that this latter map is an isomorphism, hence by rigidity $f^*f_*E \rightarrow E$ is also an isomorphism. Similarly, and again using base change, if $E_0'\subset i_c^*E$ is the largest constant sub-isocrystal, then $E':= f^*f_*E$ is a subobject of $E$ such that $i_c^*E'=E'_0$.

For the third, again I will show inductively that the canonical objects $U_n\in \isoc{X/\ekd}$ extend to objects $W_n\in \cur{N}_f\fisoc{\cur{X}/K}$, starting with $W_1=\mathcal{O}^\dagger_{\cur{X}/K}$ extending $U_1=\mathcal{O}^\dagger_{X/\ekd}$. The induction hypothesis will require in addition that:
\begin{enumerate} \item $f_*W_{n}^\vee \cong \cur{O}_{C/K}^\dagger$;
\item there exists a map $s^*W_n^\vee\rightarrow \cur{O}_{C/K}^\dagger$ such that the composite 
\[ \cur{O}_{C/K}^\dagger \cong f_*W_{n}^\vee \cong s^*(f^*f_*W_{n}^\vee)  \rightarrow s^*W_{n}^\vee \rightarrow \cur{O}_{C/K}^\dagger \]
is an isomorphism.
\end{enumerate}
I will construct $W_{n+1}$ as an extension of $W_n$ by $f^*(\mathbf{R}^1f_*W_n^\vee)^\vee $. From the five term exact sequence coming from the Leray spectral sequence for `global' absolute Frobenius cohomology (see Remark \ref{globalabs}) together with the global projection formula \cite[ \S3]{Laz15} and the induction hypothesis, there is an exact sequence
\begin{align*} 0&\rightarrow H^1_F(C/K,(\mathbf{R}^1f_*W_n^\vee)^\vee) \rightarrow \mathrm{Ext}^1_{\fisoc{\cur{X}/K}}(W_n,f^*(\mathbf{R}^1f_*W_n^\vee)^\vee) \rightarrow \mathrm{End}_{\fisoc{C/K}}(\mathbf{R}^1f_*W_n^\vee) \\ &\rightarrow H^2_F(C/K,(\mathbf{R}^1f_*W_n^\vee)^\vee) \rightarrow H^2_F(\cur{X}/K,W_n^\vee\otimes f^*(\mathbf{R}^1f_*W_n^\vee)^\vee )\rightarrow \ldots . \end{align*}
Now use the hypothesised map $s^*W_n^\vee\rightarrow \cur{O}_{C/K}^\dagger$ to split the maps
\begin{align*} H^1_F(C/K,(\mathbf{R}^1f_*W_n^\vee)^\vee) \rightarrow \mathrm{Ext}^1_{\fisoc{\cur{X}/K}}(W_n,f^*(\mathbf{R}^1f_*W_n^\vee)^\vee) \\
H^2_F(C/K,(\mathbf{R}^1f_*W_n^\vee)^\vee) \rightarrow H^2_F(\cur{X}/K,W_n^\vee\otimes f^*(\mathbf{R}^1f_*W_n^\vee)^\vee )
\end{align*}
appearing in the above sequence, and hence define $W_{n+1}$ to be the extension class 
\[[W_{n+1}]\in \mathrm{Ext}^1_{\fisoc{\cur{X}/K}}(W_n,f^*(\mathbf{R}^1f_*W_n^\vee)^\vee) \]
which maps to the identity in $\mathrm{End}_{\fisoc{C/K}}(\mathbf{R}^1f_*W_n^\vee)$ and to zero in $H^1_F(C/K,(\mathbf{R}^1f_*W_n^\vee)^\vee)$. By the base change formula for $\mathbf{R}^1f_*$, this is indeed an extension of $U_{n+1}$ to $\mathcal{N}\fisoc{\cur{X}/K}$. The completion of the proof is then exactly as in Theorem \ref{bc}. \end{proof}

\section{Good reduction criteria}

In the rest of this article, I will show how to use the `non-abelian' information contained in $\pi_1^\rig(X/\ekd,x)$ to give a criterion for a smooth, semistable curve over $F$ to have good reduction, analogous to that provided by Andreatta, Iovita and Kim in \cite{AIK15}. To formulate this criterion, I will first need to discuss the correct analogue of being `crystalline' for $\pn$-modules over $\ekd$.

Let $S_K:=\cur{V}\pow{t}\otimes_\cur{V} K$ denote the bounded power series ring over $K$, equivalently this is the `plus part' of the bounded Robba ring $\ekd$, i.e. those series $\sum_ia_it^i\in \ekd$ such that $a_i=0$ for $i<0$. Given our fixed choice of Frobenius on $S_K$, it makes sense to consider the category $\pnm{S_K}$ of $\pn$-modules over $S_K$.

\begin{proposition}[\cite{dJ98}, Theorem 1.1] The base extension functor
\[ \pnm{S_K}\rightarrow \pnm{\ekd} \]
is fully faithful.
\end{proposition}

\begin{definition} I will call an object $M\in \pnm{\ekd}$ \emph{non-singular} if it is in the essential image of this functor.
\end{definition}

The terminology is supposed to suggest that $t=0$ is a non-singular point of the differential equation. One way to think about this is that $S_K$ is a characteristic zero lift of the ring of integers $R$ of $F$. Hence if $\pn$-modules over $\ekd$ are thought of as $p$-adic local systems on $\spec{F}$, $\pn$-modules over $S_K$ should be thought of as $p$-adic local systems on $\spec{R}$. Hence $M\in \pnm{\ekd}$ is being non-singular means that it extends to a local system over $\spec{R}$; this provides at least part of the justification for this being the correct equicharacteristic analogue of a mixed charactersitic $p$-adic representation being crystalline.

Of course, this notion of `non-singularity' extends to ind-objects, and it therefore makes sense to speak of the non-abelian $\pn$-module
\[\pi_1^\rig(X/\ekd,x)=\spec{\cur{A}_{\infty,x}^{\rig,\vee}}\]
being non-singular. Using full-faithfulness (and compatibility with tensor products) of the base extension functor
\[ \pnm{S_K}\rightarrow \pnm{\ekd} \]
this is equivalent to the existence of a Hopf algebra over $\underline{\mathbf{M}\Phi}_{S_K}^\nabla$ whose base change to $\ekd$ is $\cur{A}^{\rig,\vee}_{\infty,x}$.

\begin{definition} \label{defn: stable} Let $\cur{X}\rightarrow \spec{R}$ be a flat morphism of relative dimension $1$. Then $\cur{X}$ is said to be semistable if it is \'etale locally isomorphic to either $\spec{R[x,y]/(xy-t)}$ or $\spec{R[x]}$. It is called stable if in addition it is proper over $\spec{R}$, and geometrically (i.e. over $\bar k$), any component of the special fibre isomorphic to $\P^1$ meets the other components in at least three points.
\end{definition}

\begin{remark} The conditions imply that $\cur{X}$ is regular, and that the generic fibre is smooth over $F$. Note that the definition applies equally well over any trait $S$ (i.e. the spectrum of a discrete valuation ring).
\end{remark}

The criterion for good reduction is then the following.

\begin{theorem}\label{main2} Let $\cur{X}\rightarrow \spec{R}$ be a stable curve with generic fibre $X$ geometrically connected of genus $g\geq 2$. Let $x\in X(F)$ be a rational point. Then $\cur{X}$ is smooth over $R$ if and only if the rigid fundamental group $\pi_1^\rig(X/\ekd,x)$ is non-singular.
\end{theorem}

The basic idea is to consider the special fibre as a log-smooth curve. The deformation theory of such curves is unobstructed, so this special fibre can be deformed to a stable curve in mixed characteristic. By comparing suitable monodromy operators I can then in fact deduce Theorem \ref{main2} from the mixed characteristic result in \cite{AIK15}.

Before beginning the proof properly, let me note here that it is easy to show that if $K'/K$ is a totally ramified extension, then there is a natural base change isomorphism 
\[ \pi_1^\rig(X/\cur{E}_{K'}^\dagger,x)\isomto \pi_1^\rig(X/\ekd,x) \otimes_{\ekd} \cur{E}^\dagger_{K'} \]
of $\pn$-modules over $\cur{E}^\dagger_{K'}$. To prove Theorem \ref{main2} I can therefore assume that $K=K_0$ is absolutely unramified, which I shall do for the rest of this article. In particular, $K^\sigma=\Q_q$.

\section{Monodromy for \texorpdfstring{$\pn$}{pn}-modules over \texorpdfstring{$\ekd$}{ekd}} \label{soln}

The purpose of this section is to introduce the notion of `regularity' for objects of $\pnm{\ekd}$, which is a logarithmic analogue of non-singularity, and show that for regular objects, non-singularity is equivalent to the vanishing of a certain monodromy operator. Let $\delta_t$ denote the `logarithmic' derivation $\delta_t(f)=t\frac{\partial f}{\partial t}$ on $S_K$.

\begin{definition} A logarithmic $\pn$-module over $S_K$ is a finite $S_K$-module $M$ together with
\begin{itemize} \item a logarithmic connection, that is a $K$-linear map $\nabla^{\log}:M\rightarrow M$ such that
\[ \nabla^{\log}(fm)=f\nabla^{\log}(m)+\delta_t(f)m\]
for all $f\in S_K$ and $m\in M$;
\item a horizontal Frobenius $\varphi:\sigma^*M\rightarrow M$.
\end{itemize}  The category of logarithmic $\pn$-modules over $S_K$ is denoted $\underline{\mathbf{M}\Phi}_{S_K}^{\nabla,\log}$. 
\end{definition}

The condition of admitting a Frobenius structure forces these objects to be free, by \cite[Proposition 3.1.4]{Tsu98} so it follows in the usual way that the category $\underline{\mathbf{M}\Phi}_{S_K}^{\nabla,\log}$ is in fact Tannakian over $\Q_q$. We have obvious functors
\[ \underline{\mathbf{M}\Phi}_{S_K}^{\nabla} \rightarrow \underline{\mathbf{M}\Phi}_{S_K}^{\nabla,\log} \]
sending $(M,\nabla)$ to $(M,\nabla^{\log}=t\nabla)$, and
\[\underline{\mathbf{M}\Phi}_{S_K}^{\nabla,\log} \rightarrow \underline{\mathbf{M}\Phi}_{\ekd}^{\nabla}  \]
sending $(M,\nabla^{\log})$ to $(M\otimes \ekd, \nabla=t^{-1}\nabla^{\log})$. This latter functor is fully faithful by \cite[Theorem 6.3.1]{Ked00}. 

\begin{definition} I will call objects in the essential image of $\underline{\mathbf{M}\Phi}_{S_K}^{\nabla,\log} \rightarrow \underline{\mathbf{M}\Phi}_{\ekd}^{\nabla}$ `regular'.
\end{definition}

The terminology is supposed to reflect the fact that the differential equation has regular singularities at $t=0$. The `special fibre' of such objects is then a $(\varphi,N)$-module over $K$.

\begin{definition} A $(\varphi,N)$-module over $K$ is a finite dimensional $K$-vector space $V$ together with a $\sigma$-linear bijective map $\varphi:V\rightarrow V$ and a nilpotent linear map $N:V\rightarrow V$ such that $N\varphi=q\varphi N$. The category of $(\varphi,N)$-modules over $K$ is denoted by $\underline{\mathbf{M}\Phi}_K^N$. 
\end{definition}

By taking the `residue' of the connection at $t=0$ we obtain a functor
\[\underline{\mathbf{M}\Phi}_{S_K}^{\nabla,\log} \rightarrow  \underline{\mathbf{M}\Phi}_K^N, \]
the presence of Frobenius structures again ensures that this residue is indeed nilpotent. In particular, we can associate a $(\varphi,N)$-module $V$ to any \emph{regular} $\pn$-module over $\ekd$.

\begin{proposition} Let $M\in \pnm{\ekd}$ be regular. Then $M$ is non-singular if and only if the associated monodromy operator $N$ vanishes.
\end{proposition}

\begin{proof} One direction is clear, so suppose for the converse that the monodromy operator $N$ associated to some $(M,\nabla^{\log})\in \underline{\mathbf{M}\Phi}_{S_K}^{\nabla,\log}$ vanishes. Then explicitly what this says is that for all $m\in M$, $\nabla^{\log}(m)\in tM$. Since $M$ is free as an $S_K$-module, we may define a map $\nabla:M\rightarrow M$ by simply setting $\nabla(m)=t^{-1}\nabla^{\log}(m)$. Then one easily checks that $\nabla(fm)=f\nabla(m)+\frac{\partial f}{\partial t}\cdot m$, so that $\nabla$ is a standard, non-logarithmic connection on $M$. Moreover the compatibility of the Frobenius $\varphi$ with $\nabla^{\log}$ implies its compatibility with $\nabla$. In other words $(M,\nabla)$ is an object of $\pnm{S_K}$ whose associated log-$\pn$-module is exactly $(M,\nabla^{\log})$.
\end{proof}

\section{Regularity of \texorpdfstring{$\pi_1^\rig(X,x)$}{pi1rig} in the semistable case} \label{sec: ssreg}

In the previous section I introduced the notion of regularity for $\pn$-modules over $\ekd$. The same definition can be extended to ind-$\pn$-modules, hence it makes sense to ask whether or not the fundamental group $\pi_1^\rig(X/\ekd,x)$ of some geometrically connected $F$-variety $X$ at some rational point $x\in X(F)$ is regular. As expected, this will be true in the semistable case.

\begin{theorem} \label{theo: ssreg}
Let $f:\cur{X}\rightarrow \spec{R}$ be proper and semistable, with geometrically connected generic fibre $X$, and let $x\in X(F)$. Then the rigid fundamental group $\pi_1^\rig(X/\ekd,x)$ is regular.
\end{theorem}

Since $\cur{X}$ is semistable, there is a canonical log structure $M$ on $\cur{X}$ coming from the special fibre. Consider the category of log-$F$-isocrystals on $(\cur{X},M)$ relative to $K$, defined as follows. Let $((\cur{X},M)/W)_\mathrm{cris}^{\log}$ denote the log crystalline site of $(\cur{X},M)$ relative to $W$ with the trivial log structure, as defined for example in \cite{Kat89}. Thus objects consists of triples $(U,(T,N),\delta)$ where $U$ is \'etale over $X$, $(T,N)$ is an exact nilpotent thickening of $(U,M)$ over $\spec{W}$ and $\delta$ is a PD structure on the ideal of $U$, compatible with the canonical PD structure on $(p)\subset W$. Let $\cur{O}_{({\cur{X},M)/W}}$ denote the sheaf of rings on $((\cur{X},M)/W)_\mathrm{cris}^{\log}$ defined by $(U,(T,N),\delta)\mapsto \Gamma(T,\cur{O}_T)$.  

Let $\mathrm{Cris}((\cur{X},M)/W)$ denote the category of finitely presented crystals of $\cur{O}_{({\cur{X},M)/W}}$-modules. Let $\mathrm{Isoc}((\cur{X},M)/K)$ denote the isogeny category $\mathrm{Cris}((\cur{X},M)/W)_{\Q}$, this is functorial in $(\cur{X},M)$ and $K$. Hence there is a Frobenius pullback functor
\[ F^*:\mathrm{Isoc}((\cur{X},M)/K)\rightarrow \mathrm{Isoc}((\cur{X},M)/K) \]
and it makes sense to speak of Frobenius structures on objects in $\mathrm{Isoc}((\cur{X},M)/K)$. Let $F\text{-}\mathrm{Isoc}((\cur{X},M)/K)$ denote the category of isocrystals together with a Frobenius structure. Let $\cur{O}_{({\cur{X},M)/K}}$ denote $\cur{O}_{({\cur{X},M)/W}}$ viewed as an object of $\mathrm{Isoc}((\cur{X},M)/K)$ or $F\text{-}\mathrm{Isoc}((\cur{X},M)/K)$.

Denote by $L$ the log structure on $\spec{R}$ coming from the closed point, then the analogous category $F\text{-}\mathrm{Isoc}((\spec{R},L)/K)$ can be defined similarly, and using the Frobenius lift $\sigma$ on $S_K$ there is a fully faithful realisation functor
\[ F\text{-}\mathrm{Isoc}((\spec{R},L)/K) \rightarrow  \underline{\mathbf{M}\Phi}_{S_K}^{\nabla,\log}. \]

\begin{lemma} \label{eqiov} This functor is an equivalence of categories.
\end{lemma}

\begin{proof} It suffices to observe that the presence of a Frobenius structure enforces quasi-nilpotence of the connection \cite[Proposition 3.4.2]{Tsu98}.
\end{proof}

By locally lifting $\cur{X}$ to affine schemes of characteristic $0$ admitting Frobenius lifts, realising on these open affines, applying $j^\dagger$, and then gluing, there is a functor
\[ j^\dagger: F\text{-}\mathrm{Isoc}((\cur{X},M)/K)\rightarrow \fisoc{X/K}. \]
The presence of Frobenius means that overconvergence conditions are automatic, by \cite[Proposition 3.25]{LP16}.

Now, for any $E\in F\text{-}\mathrm{Isoc}((\cur{X},M)/K)$ the relative log-crystalline cohomology groups 
\[ H^i_\mathrm{\log\text{-}\mathrm{cris}}((\cur{X},M)/S_K,E)\]
(see for example \cite[\S2]{HK94}) can be considered as objects of $\underline{\mathbf{M}\Phi}_{S_K}^{\nabla,\log}$, there are also absolute cohomology groups $H^i_\mathrm{\log\text{-}\mathrm{cris}}((\cur{X},M)/K,E)$, which are (not necessarily finite dimensional) vector spaces over $K$ together with a (not necessarily bijective) $\sigma$-linear endomorphism. By promoting these latter cohomology groups to complexes of $K$-modules with a weak Frobenius, it is possible, exactly as in \S\ref{safc}, to define `absolute' Frobenius cohomology groups in this context (as vector spaces over $\Q_q$, possibly infinite dimensional), satisfying all the expected formal properties. Using the zeroeth relative and absolute cohomology groups, I can now prove the following.

\begin{proposition} \label{prop: jdagff} The functor $ j^\dagger: F\text{-}\mathrm{Isoc}((\cur{X},M)/K)\rightarrow \fisoc{X/K}$ is fully faithful.
\end{proposition}

\begin{proof} Given $E,E'\in F\text{-}\mathrm{Isoc}((\cur{X},M)/K)$ there is an internal hom object
\[ \cur{H}om(E,E')\in F\text{-}\mathrm{Isoc}((\cur{X},M)/K)\]
such that 
\begin{align*} \mathrm{Hom}_{F\text{-}\mathrm{Isoc}((\cur{X},M)/K)}(E,E') &\cong \mathrm{Hom}_{F\text{-}\mathrm{Isoc}((\cur{X},M)/K)}(\cur{O}_{({\cur{X},M)/K}},\cur{H}om(E,E')) \\&\cong H^0_{\log\text{-}\mathrm{cris}}((\cur{X},M)/K, \cur{H}om(E,E'))^{\varphi=\mathrm{id}},
\end{align*}
moreover the group $H^0_{\log\text{-}\mathrm{cris}}	((\cur{X},M)/K, \cur{H}om(E,E'))^{\varphi=\mathrm{id}}$ can be identified with
\[ \mathrm{Hom}_{\underline{\mathbf{M}\Phi}_{S_K}^{\nabla,\log}}(S_K,H^0_{\log\text{-}\mathrm{cris}}((\cur{X},M)/S_K, \cur{H}om(E,E')))  .\]
Now, formation of internal homs commutes with
\[ j^\dagger:F\text{-}\mathrm{Isoc}((\cur{X},M)/K)\rightarrow \fisoc{X/K},\]
and there is a similar calculation computing hom groups in the latter category. Moreover, for any $E\in F\text{-}\mathrm{Isoc}((\cur{X},M)/K)$ the base change map
\[ H^0_{\log\text{-}\mathrm{cris}}((\cur{X},M)/S_K, E) \otimes_{S_K} \ekd \isomto H^0_\rig(X/\ekd,j^\dagger E) \]
is an isomorphism (as can be seen for example by base change to $\ek$ and applying proper base change for log-crystalline cohomology). Hence the required full faithfulness theorem follows from the corresponding claim for 
\[ \underline{\mathbf{M}\Phi}_{S_K}^{\nabla,\log} \rightarrow \pnm{\ekd}, \]
i.e. \cite[Theorem 6.3.1]{Ked00}.
\end{proof}

\begin{corollary} The category $F\text{-}\mathrm{Isoc}((\cur{X},M)/K)$ is Tannakian over $\Q_q$.
\end{corollary}

\begin{remark} Normally to obtain Tannakian categories of `logarithmic' objects, one needs to impose some sort of condition of having `nilpotent residues'. The point is that the presence of Frobenius structures forces objects to have nilpotent residues.
\end{remark}

Let $\cur{N}_fF\text{-}\mathrm{Isoc}((\cur{X},M)/K)$ denote the category of relatively unipotent objects in $F\text{-}\mathrm{Isoc}((\cur{X},M)/K)$, i.e. those which are iterated extensions of those pulled back from $\underline{\mathbf{M}\Phi}_{S_K}^{\nabla,\log}\cong F\text{-}\mathrm{Isoc}((\spec{R},L)/K)$ via $f^*$. Since $f$ is proper, the smooth point $x:\spec{F}\rightarrow X$ extends uniquely to a section $s:(\spec{R},L)\rightarrow (\cur{X},M)$ of $f$. Associated to the pair of functors
\[ s^*:\cur{N}_fF\text{-}\mathrm{Isoc}((\cur{X},M)/K)  \leftrightarrows \underline{\mathbf{M}\Phi}_{S_K}^{\nabla,\log} :f^*\]
there is therefore an affine group scheme $\pi_1^{\log\text{-}\mathrm{cris}}((\cur{X},M)/S_K,s)$ over $\underline{\mathbf{M}\Phi}_{S_K}^{\nabla,\log}$ as in \S\ref{fgtc}. The commutative diagram
\[  \xymatrix{\cur{N}_fF\text{-}\mathrm{Isoc}((\cur{X},M)/K) \ar[r]^-{j^\dagger}\ar@/^1pc/[d]^{s^*} & \cur{N}_\mathrm{rel}\fisoc{X/K}  \ar@/^1pc/[d]^{x^*} \\ \underline{\mathbf{M}\Phi}_{S_K}^{\nabla,\log} \ar[r]^{-\otimes \ekd} \ar[u]^{f^*} & \underline{\mathbf{M}\Phi}^{\nabla}_{\ekd} \ar[u]^{-\otimes \cur{O}_{X/K}^\dagger}  } \]
induces a homomorphism
\[ \pi_1^\rig(X/\ekd,x)\rightarrow \pi_1^{\log\text{-}\mathrm{cris}}((\cur{X},M)/S_K,s)\otimes_{S_K} \ekd\]
of affine group schemes over $\pnm{\ekd}$. Now the proof of Theorem \ref{theo: ssreg} boils down to the following.

\begin{proposition}
The map $\pi_1^\rig(X/\ekd,x)\rightarrow \pi_1^{\log\text{-}\mathrm{cris}}((\cur{X},M)/S_K,s)\otimes_{S_K} \ekd$ is an isomorphism.
\end{proposition}

\begin{proof}  As in the proof of Theorem \ref{cgt}, it suffices to show this after `forgetting the $\pn$-module structures' on both sides, i.e. to show that the base change morphism
\[ \pi_1^\rig(X/\ekd,x)\rightarrow \pi_1^{\log\text{-}\mathrm{cris}}((\cur{X},M)/S_K,s)\otimes_{S_K} \ekd\]
associated to the commutative diagram
\[  \xymatrix{ \cur{N}_fF\text{-}\mathrm{Isoc}((\cur{X},M)/K) \ar[r]^-{j^\dagger}\ar@/^1pc/[d]^{s^*} & \cur{N}\isoc{X/\ekd}  \ar@/^1pc/[d]^{x^*} \\ \underline{\mathbf{M}\Phi}_{S_K}^{\nabla,\log} \ar[r] \ar[u]^{f^*} & \mathrm{Vec}_{\ekd} \ar[u]^{-\otimes \cur{O}_{X/\ekd}^\dagger}  } \]
is an isomorphism. This is proved more or less exactly as Theorem \ref{cgt}. The role of the pushforwards $\mathbf{R}^if_*$ is here played by the relative log crystalline cohomology
\[ H^i_{\log\text{-}\mathrm{cris}} ((\cur{X},M)/S_K,E)\in \underline{\mathbf{M}\Phi}_{S_K}^{\nabla,\log}  \]
 and that of absolute Frobenius cohomology by the entirely similar version of absolute Frobenius cohomology for $(\cur{X},M)$ (which I indicated how to construct in the paragraph preceding Proposition \ref{prop: jdagff}). The all-important base change map then becomes the isomorphism
\[ H^i_{\log\text{-}\mathrm{cris}} ((\cur{X},M)/S_K,E)  \otimes_{S_K} \ekd \cong H^i_\rig(X/\ekd,j^\dagger E)\]
of vector spaces over $\ekd$. This is again verified by base changing to $\ek$ and applying proper base change in log-crystalline cohomology.
\end{proof}

Thus the log-crystalline fundamental group $\pi_1^{\log\text{-}\mathrm{cris}}((\cur{X},M)/S_K,s)$ is a non-abelian log-$\pn$-module over $S_K$ giving rise to $\pi_1^\rig(X/\ekd,x)$ upon base change to $\ekd$, so $\pi_1^\rig(X/\ekd,x)$ is regular. Moreover, the associated monodromy operator $N$ can be described purely in terms of the log special fibre of $\cur{X}$ as follows.

Let $X_0$ denote the special fibre of the family $\cur{X}$, and $M_0$ the log-structure on $X_0$ obtained by pulling back $M$ via $X_0\rightarrow \cur{X}$. Let $W(L_0)$ denote the Teichm\"uller lift to $W$ of the log structure $L_0$ on $k$ of the punctured point. Consider the following two categories:
\begin{enumerate}
\item the category
\[ F\text{-}\mathrm{Isoc}((X_0,M_0)/K) =  F\text{-}\mathrm{Crys}((X_0,M_0)/W)_{\Q}  \]
of log-$F$-isocrystals on $(X_0,M_0)/W$, relative to the trivial log structure on $W$;
\item the isogeny category $\mathrm{Crys}((X_0,M_0)/(W,W(L_0)))_{\Q}$ of log-crystals on $(X_0,M_0)$ relative to the log structure $W(L_0)$ on $W$.
\end{enumerate}

\begin{remark} In the first case, the notation should be interpreted as $F\text{-}(\mathrm{Crys}_{\Q})$ and not $(F\text{-}\mathrm{Crys})_{\Q}$.
\end{remark}

Note that the corresponding categories $F\text{-}\mathrm{Isoc}((\spec{k},L_0)/K)$ and $\mathrm{Crys}((\spec{k},L_0)/(W,W(L_0))_{\Q}$ are simply the categories $\underline{\mathbf{M}\Phi}_K^N$ of $(\varphi,N)$-modules and $\mathrm{Vec}_K$ of finite dimensional vector spaces over $K$ respectively. The subcategory
\[ \cur{N}\mathrm{Crys}((X_0,M_0)/(W,W(L_0)))_{\Q} \subset \mathrm{Crys}((X_0,M_0)/(W,W(L_0)))_{\Q}\]
of unipotent objects  is Tannakian, and 
\[ s_0^*: \cur{N}\mathrm{Crys}((X_0,M_0)/(W,W(L_0)))_{\Q} \rightarrow \mathrm{Vec}_K \]
provides a fibre functor.

\begin{definition} Define $\pi_1^{\log\text{-}\mathrm{cris}}((X_0,M_0)/K,s_0)$ to be the associated affine group scheme over $K$.
\end{definition}

Next, let
\[ \cur{N}_{f_0}F\text{-}\mathrm{Isoc}((X_0,M_0)/K) \subset F\text{-}\mathrm{Isoc}((X_0,M_0)/K) \]
denote the category of relatively unipotent objects i.e. those which are iterated extensions of those pulled back from $\underline{\mathbf{M}\Phi}_{K}^{N}\cong F\text{-}\mathrm{Isoc}((\spec{k},L_0)/K)$ via $f_0^*$, then again associated to the pair of functors
\[ s_0^*:\cur{N}_{f_0}F\text{-}\mathrm{Isoc}((X_0,M_0)/K)  \leftrightarrows \underline{\mathbf{M}\Phi}_{K}^{N} :f_0^*\]
there is an affine group scheme over $\underline{\mathbf{M}\Phi}_{K}^{N}$ which for now I will denote by $G^{\log\text{-}\cris}_{(X_0,M_0),s_0}$. There is a commutative diagram
\[  \xymatrix{\cur{N}_fF\text{-}\mathrm{Isoc}((\cur{X},M)/K) \ar[r] \ar@/^1pc/[d]^{s^*} & \cur{N}_{f_0}F\text{-}\mathrm{Isoc}((X_0,M_0)/K)   \ar@/^1pc/[d]^{s_0^*} \ar[r] &\cur{N}\mathrm{Crys}((X_0,M_0)/(W,W(L_0)))_{\Q}  \ar@/^1pc/[d]^{s_0^*}  \\ \underline{\mathbf{M}\Phi}_{S_K}^{\nabla,\log} \ar[r]^{-\otimes_{S_K} K } \ar[u]^{f^*} & \underline{\mathbf{M}\Phi}^{N}_{K} \ar[u]^{f_0^*}  \ar[r] & \mathrm{Vec}_K, \ar[u]^{f_0^*}  } \]
where the horizontal arrows in the right hand square are simply the forgetful functors. This induces a homomorphism
\[G^{\log\text{-}\cris}_{(X_0,M_0),s_0} \rightarrow \pi_1^{\log\text{-}\mathrm{cris}}((\cur{X},M)/S_K,s)\otimes_{S_K} K \]
or affine group schemes over $\underline{\mathbf{M}\Phi}_K^N$, and homomorphisms
\begin{align*}
\pi_1^{\log\text{-}\mathrm{cris}}((X_0,M_0)/K,s_0) &\rightarrow \pi_1^{\log\text{-}\mathrm{cris}}((\cur{X},M)/S_K,s)\otimes_{S_K} K \\
\pi_1^{\log\text{-}\mathrm{cris}}((X_0,M_0)/K,s_0) &\rightarrow G^{\log\text{-}\cris}_{(X_0,M_0),s_0}
\end{align*}
of affine group schemes over $K$, such that the obvious triangle commutes.

\begin{proposition} All these homomorphisms are isomorphisms. 
\end{proposition}

Of course, the proof exactly follows the previous models, using proper base change in log crystalline cohomology. I will henceforth write $\pi_1^{\log\text{-}\mathrm{cris}}((X_0,M_0)/K,s_0) $ instead of $G^{\log\text{-}\cris}_{(X_0,M_0),s_0}$. The upshot, then, of all of this, is that the monodromy operator associated to the regular non-abelian $\pn$-module $\pi_1^\rig(X/\ekd,x)$ is simply the monodromy operator on the non-abelian $(\varphi,N)$-module $\pi_1^{\log\text{-}\mathrm{cris}}((X_0,M_0)/K,s_0)$.

\section{Deformations to mixed characteristic}

What the results of the previous section mean is that I can now completely forget about the semistable family $\cur{X}\rightarrow \spec{R}$, and concentrate solely on the special fibre as a `log-smooth stable curve' over $k$. Again, I will denote by $L_0$ the log structure of the punctured point on $k$.

\begin{definition} Let $(X_0,M_0)\rightarrow (\spec{k},L_0)$ be a finite type log-smooth morphism of log schemes. We say that $(X_0,M_0)$ is a semistable log curve over $k$ if \'etale locally it is strict \'etale over either
\begin{itemize}
\item $\spec{k[x,y]/(xy)}$ with the obvious log structure $\N^2\rightarrow k[x,y]/(xy) $, or;
\item $\spec{k[x]}$ with the `punctured' log structure $\N\rightarrow k[x]$, $1\mapsto 0$.
\end{itemize}
We say that $(X_0,M_0)$ is a stable log curve if it is moreover proper, and geometrically (i.e. over $\bar k$) every component isomorphic to $\mathbb{P}^1$ meets the other components in at least three points.
\end{definition}

Let $(X_0,M_0)$ be a geometrically connected stable log curve over $(k,L_0)$, and $s_0:(k,L_0)\rightarrow (X_0,M_0)$ a section of the structure morphism. Then as in \S\ref{sec: ssreg} above its unipotent fundamental group can be considered as a non-abelian $(\varphi,N)$-module $\pi_1^{\log\text{-}\mathrm{cris}}((X_0,M_0)/K,s_0)$. By the results of \S\ref{soln} and \S\ref{sec: ssreg}, the proof of Theorem \ref{main2} amounts to showing that if the monodromy operator $N$ on $\pi_1^{\log\text{-}\mathrm{cris}}((X_0,M_0)/K,s_0)$ vanishes, then $X_0$ is in fact smooth over $k$ (rather than just log-smooth). A completely algebraic proof of this fact is the subject of current work in progress of Chiarellotto, di\thinspace Proietto and Shiho \cite{CDPS}, but instead I will deduce it from the mixed characteristic result in \cite{AIK15}.

Let $L'$ denote the log structure on $\spec{W}$ coming from the closed point. Since the deformation theory of log-smooth curves is unobstructed \cite[Proposition 8.6]{Kat96}, it follows that there exists a log-smooth and proper curve $g:(\cur{Y},N)\rightarrow (\spec{W},L')$ with special fibre $(X_0,M_0)$. Thus $\cur{Y}$ is a stable curve over $W$ (in the sense of Definition \ref{defn: stable}), $N$ is the log structure associated to the special fibre $X_0$, and $M_0=N|_{X_0}$.

Let $Y/K$ denote the generic fibre of $\cur{Y}$, and lift the section $s_0$ to a section $t$ of $g$. Let $y\in Y(K)$ denote the corresponding point. It follows from \cite[Theorems 1.4, 1.6]{AIK15} that:
\begin{enumerate}
\item the non-abelian Galois representation $\pi_1^\et(Y_{\overline{K}},y)_{\Q_p}$ is semistable;
\item $\cur{Y}\rightarrow \spec{W}$ is smooth if and only if $\pi_1^\et(Y_{\overline{K}},y)_{\Q_p}$ is crystalline.
\end{enumerate}
Now let
\[ \mathbf{D}_\mathrm{st}: \mathrm{Rep}^\mathrm{st}_{\Q_p}(G_K) \rightarrow \underline{\mathbf{M}\Phi}_K^N \]
be Fontaine's functor taking a semistable Galois representation to its associated $(\varphi,N)$-module (forgetting the Hodge filtration $F^\bu$). It is well-known that $V \in \mathrm{Rep}^\mathrm{st}_{\Q_p}(G_K)$ is crystalline if and only if the monodromy operator $N$ on $\mathbf{D}_\mathrm{st}(V)$ is zero. Hence to complete the proof of Theorem \ref{main2} it suffices to show the following.

\begin{theorem} \label{theo: dst}
There is an isomorphism
\[ \mathbf{D}_\mathrm{st}(\pi_1^\et(Y_{\overline{K}},y)_{\Q_p}) \cong \pi_1^{\log\text{-}\mathrm{cris}}((X_0,M_0)/K,s_0) \]
of affine group schemes over $\underline{\mathbf{M}\Phi}_K^N$.
\end{theorem}

To prove this, I will need to give a fairly concrete description of $\mathbf{D}_\mathrm{st}(\pi_1^\et(Y_{\overline{K}},y)_{\Q_p})$ in terms of the `universal unipotent objects' $U_n$ used both in \cite{AIK15} and in the proof of the various base change theorems above. These objects were originally introduced by Hadian \cite{Had11}, and they also play a key role in the algebraic approach to `good reduction' theorems for curves in \cite{CDPS}. 

In the context of the log special fibre $(X_0,M_0)$ these can be described as follows. As before, let $W(L_0)$ denote the Teichm\"uller lift to $W$ of the log structure of the punctured point on $\spec{k}$, and consider the category $\mathrm{Crys}((X_0,M_0)/(W,W(L_0)))_{\Q}$ of  log-isocrystals on the \emph{relative} log-crystalline site. A pointed object
\[ (E,e)\in \cur{N}\mathrm{Crys}((X_0,M_0)/(W,W(L_0)))^*_{\Q}\]
will be an object $E$ together with a chosen element $e \in s_0^*E$.

\begin{proposition} There exists a projective system
\[ \{(U_n,u_n)\}_{n\geq1}\in \cur{N}\mathrm{Crys}((X_0,M_0)/(W,W(L_0)))^*_{\Q}\]
of pointed objects, which is universal in the sense that for any other pointed object $(E,e)$, with $E$ of unipotence degree $\leq n$, there exists a unique morphism
\[ U_n \rightarrow E \]
mapping $u_n$ to $e$.
\end{proposition}

\begin{remark} Of course, this property guarantees uniqueness of the projective system $\{(U_n,u_n)\}_{n \geq 1}$ of pointed isocrystals.
\end{remark}

\begin{proof} See \cite[\S2]{Had11}, \cite[\S3]{AIK15} or the proof of Theorem \ref{bc} above.
\end{proof}

Following the proof of Theorem \ref{bc}, these can then be extended canonically to objects
\[ W_n\in \cur{N}_{f_0}F\text{-}\mathrm{Isoc}((X_0,M_0)/K),\]
and the points $u_n\in s_0^*U_n$ to morphisms $w_n:K \rightarrow s_0^*W_n$ in $\underline{\mathbf{M}\Phi}_K^N$. The corresponding universal property is then the following.

\begin{proposition} For any $E\in \cur{N}_{f_0}F\text{-}\mathrm{Isoc}((X_0,M_0)/K)$ of relative unipotence degree $\leq n$, the induced map
\[ H^0_{\log\text{-}\mathrm{cris}}((X_0,M_0)/K,\cur{H}om(W_n,E))\rightarrow s_0^*E \]
given by composing with $w_n$ is an isomorphism in $\underline{\mathbf{M}\Phi}_K^N$.
\end{proposition}

\begin{remark} Here $\cur{H}om(W_n,E)$ is internal hom in the category $F\text{-}\mathrm{Isoc}((X_0,M_0)/K)$, and the cohomology groups $H^0_{\log\text{-}\mathrm{cris}}((X_0,M_0)/K, -)$ are understood to be Hyodo--Kato cohomology groups taking values in the category of $(\varphi,N)$-modules.
\end{remark}

\begin{proof}
The point is that to check that the induced map
\[ H^0_{\log\text{-}\mathrm{cris}}((X_0,M_0)/K,\cur{H}om(W_n,E) )\rightarrow s_0^*E \]
is an isomorphism, it suffices to do so after taking the `fibre', i.e. forgetting the $(\varphi,N)$-module structure. But now it just follows from the corresponding universal property of the $U_n$.
\end{proof}

As in \cite[\S2]{Had11}, the universal properties of the $W_n$ then guarantee the existence of maps
\[ \cur{O}_{(X_0,M_0)/K}=W_1 \rightarrow W_n,\;\; W_{n+m}\rightarrow W_n \otimes W_m \]
such that $1\mapsto w_n$, and $w_{n+m}\mapsto w_n\otimes w_m$ respectively. Moreover, the identification
\[ s_0^*W_n \cong H^0_{\log\text{-}\mathrm{cris}}((X_0,M_0)/K,\cur{E}nd(W_n)) \]
induces an algebra structure on each $W_n$. All of these then combine to make
\[ A_{\infty,s_0}^{\cris,\vee}:= \mathrm{colim}_n (s_0^*W_n)^\vee \]
into a Hopf $\underline{\mathbf{M}\Phi}_K^N$-algebra in the sense of \S\ref{fgtc}. The proof of the following result is given in \cite[\S6]{AIK15}.

\begin{theorem}[\cite{AIK15}, Theorem 1.8] There is an isomorphism
\[\mathbf{D}_\mathrm{st}(\pi_1^\et(Y_{\overline{K}},y)_{\Q_p}) \cong \spec{A_{\infty,s_0}^{\cris,\vee}} \]
of non-abelian $(\varphi,N)$-modules over $K$.
\end{theorem}

The proof of Theorem \ref{theo: dst}, and hence Theorem \ref{main2}, can now be completed as follows.

\begin{proof}[Proof of Theorem \ref{theo: dst}]
It suffices to show that
\[ \spec{A_{\infty,s_0}^{\cris,\vee}}\cong\pi_1^{\log\text{-}\mathrm{cris}}((X_0,M_0)/K,s_0),\]
and the proof of this is very similar to the proof of similar results in \cite{CDPS}. To construct a homomorphism between them, it suffices by \cite[Proposition 2.3]{Laz15} together with the defining property of $\pi_1^{\log\text{-}\mathrm{cris}}((X_0,M_0)/K,s_0)$ to produce a functor from the category $\cur{N}_{f_0}F\text{-}\mathrm{Isoc}((X_0,M_0)/K)$ to the category of representations of $\spec{A_{\infty,s_0}^{\cris,\vee}}$, again as an affine group scheme over $\underline{\mathbf{M}\Phi}_K^N$.

This can be achieved once more by the universal properties of the $W_n$: if we have any relatively unipotent $E\in  \cur{N}_{f_0}F\text{-}\mathrm{Isoc}((X_0,M_0)/K)$ of relative unipotence degree $\leq n$, then the action of $\cur{E}nd(W_n)$ on $\cur{H}om(W_n,E)$ makes
\[ s_0^*E = H^0_{\log\text{-}\mathrm{cris}}((X_0,M_0)/K,\cur{H}om(W_n,E))  \]
into a module over
\[ s_0^*W_n = H^0_{\log\text{-}\mathrm{cris}}((X_0,M_0)/K,\cur{E}nd(W_n)),  \]
compatibly as $n$ varies. Therefore $s_0^*E$ becomes a co-module over $A_{\infty,s_0}^{\cris,\vee}$, in other words a representation of $\spec{A_{\infty,s_0}^{\cris,\vee}}$. Finally to prove that this map is an isomorphism, it suffices to do so after forgetting the extra $(\varphi,N)$-module structure, i.e. when considering them simply as affine group schemes over $K$. This can then be shown word for word as in \cite[Theorem 2.9]{Had11}.
\end{proof}

In fact, using the more refined version of \cite[Theorem 1.6]{AIK15}, namely \cite[Proposition 6.3]{AIK15}, allows the deduction of a correspondingly more refined version of Theorem \ref{main2}. Number the lower central series of a unipotent group $U$ by $U[1]=U$, $U[n]=[U[n-1],U]$ and write $U^{(n)}$ for $U/U[n]$.

\begin{theorem} \label{main4} Let $X/F$ be a smooth and proper curve with semistable reduction, and fix $x\in X(F)$. Then $X$ has good reduction if and only if the non-abelian $\pn$-module $\pi_1^\rig(X/\ekd,x)^{(4)}$ is solvable.
\end{theorem}

\bibliographystyle{../../Templates/bibsty}
\bibliography{../../lib.bib}

\end{document}